\documentclass[12pt]{amsart}
\usepackage{tikz-cd,enumerate,ulem}
\usepackage{mathtools}
\usetikzlibrary{arrows}
\usepackage[utf8]{inputenc}
\usepackage{xspace}

\usepackage[english]{babel}
\usepackage[T1]{fontenc}
\usepackage{amsmath,amssymb}

\usepackage[a4paper,top=3cm,bottom=2cm,left=3cm,right=3cm,marginparwidth=1.75cm]{geometry}

\usepackage{amsthm}
\usepackage{amsfonts}
\usepackage{graphicx}
\usepackage[colorinlistoftodos]{todonotes}
\usepackage[colorlinks=true, allcolors=blue]{hyperref}
\usepackage{cleveref}
\usepackage{url}
\newtheorem{theorem}{Theorem}[section]
\newtheorem{thm}[theorem]{Theorem}

\newtheorem{proposition}[theorem]{Proposition}
\newtheorem{definition}[theorem]{Definition}
\newtheorem{lemma}[theorem]{Lemma}

\newtheorem{corollary}[theorem]{Corollary}

\newtheorem{prop}[theorem]{Proposition}
\newtheorem{cor}[theorem]{Corollary}

\theoremstyle{definition}
\newtheorem{ex}[theorem]{Example}
\newtheorem{rem}[theorem]{Remark}
\newtheorem{remark}[theorem]{Remark}

{
      \theoremstyle{plain}
      
  }
  {
      \theoremstyle{plain}
      
  }

\def\FF{\mathbb{F}}
\def\GG{\mathbb{G}}
\def\KK{\mathbb{K}}
\def\LL{\mathbb{L}}
\def\NN{\mathbb{N}}
\def\QQ{\mathbb{Q}}

\def\ZZ{\mathbb{Z}}

\def\calo{\mathcal{O}}

\def\calS{\mathcal{S}}
\def\calY{\mathcal{Y}}

\def\End{\mathrm{End}}

\def\Gal{\mathrm{Gal}}
\def\GL{\mathrm{GL}}
\def\irr{{\mathrm{irr}}}
\def\PGL{\mathrm{PGL}}

\def\SL{\mathrm{SL}}
\def\red{{\mathrm{red}}}

\def\univ{{\mathrm{univ}}}
\def\Qbar{\overline{\QQ}}

\newcommand{\VV}{\mathbb{V}}
\newcommand{\Ind}{\mathrm{Ind}}

\newcommand{\mug}{{\boldsymbol\mu}}

\DeclareMathOperator{\Spec}{Spec}

\newtheorem{innercustomthm}{Theorem}
\newenvironment{customthm}[1]
  {\renewcommand\theinnercustomthm{#1}\innercustomthm}
  {\endinnercustomthm}
\newtheorem{innercustomcor}{Corollary}
\newenvironment{customcor}[1]
  {\renewcommand\theinnercustomcor{#1}\innercustomcor}
  {\endinnercustomcor}

\usepackage{enumitem}
\setlist[itemize,1]{label={--\,}}
\setlist{leftmargin=8mm}

\newtheorem*{fact*}{Fact} 
\newcommand{\Z}{{\mathbb{Z}}}
\newcommand{\Zp}{{\overline{\Z}_p}}
\newcommand{\Q}{{\mathbb{Q}}}
\newcommand{\Qp}{\overline\Q_p}

\newcommand{\F}{{\mathbb{F}}}

\newcommand{\fg}{{\mathfrak g}}
\newcommand{\fgl}{{\mathfrak{gl}}}
\newcommand{\fh}{{\mathfrak h}}

\newcommand\fp{{\mathfrak{p}}}
\newcommand{\ft}{{\mathfrak t}}
\newcommand{\fsl}{{\mathfrak{sl}}}

\newcommand\B{{\rm (B)}\xspace}

\newcommand\CE{{\eqref{eq:H1}}\xspace}
\newcommand\LPTR{{\ref{def:lptr}}\xspace}
\newcommand\Mat{{\mathrm{M}}}
\newcommand\NC{{\ref{eq:NC}}\xspace} 
\newcommand\PTR{{\ref{def:ptr}}\xspace}
\newcommand\N{{N}} 

\newcommand\cE{{\mathcal{E}}}

\newcommand\cO{{\mathcal{O}}}

\DeclareMathOperator{\im}{im} 
\DeclareMathOperator{\trace}{tr}
\newcommand\ab{{\mathrm{ab}}}
\newcommand\CNL{{\mathrm{CNL}}}

\newcommand\id{{\mathrm{id}}}
\newcommand\Id{{\mathrm{Id}}}
\newcommand\sq{{\square}}
\newcommand\sms{{\mathrm{ss}}}

\newcommand\ind{{\mathrm{ind}}}
\newcommand\notind{{\mathrm{nind}}}
\newcommand\loc{{\mathrm{loc}}}

\newcommand{\vareps}{\varepsilon}

\newcommand{\into}{\hookrightarrow}
\newcommand{\onto}{\twoheadrightarrow}
\newcommand{\ovl}{\overline}
\newcommand{\wtl}{\widetilde}
\newcommand{\wh}{\widehat}
\newcommand{\ccirc}{\kern0.5ex\vcenter{\hbox{$\scriptstyle\circ$}}\kern0.5ex}

\usepackage{crossreftools}

\makeatletter
\newcommand{\optionaldesc}[2]{%
  \phantomsection
  #1\protected@edef\@currentlabel{#1}\label{#2}%
}
\makeatother

\newtheorem*{theorem*}{Theorem}
\newtheorem*{remark*}{Remark}

\numberwithin{equation}{section}
\author{Andrea Conti and Lea Terracini}
\title{Bogomolov property for Galois representations with big local image}
\keywords{Weil height, Bogomolov property, $p$-adic Lie groups, images of Galois representations, modular forms}
\subjclass[2020]{11G50; 11F80}

\begin{document}

\begin{abstract} An algebraic extension of the rational numbers is said to have the \textit{Bogomolov property} \B if the absolute logarithmic Weil height of its non-torsion elements is uniformly bounded from below. Given a continuous representation $\rho$ of the absolute Galois group $G_\KK$ of a number field $\KK$, one says that $\rho$ has \B if the subfield of $\ovl\Q$ fixed by $\ker(\rho)$ has \B. We prove that, if $\rho:G_\KK \to \GL_d(\ZZ_p)$ maps an inertia subgroup at a prime above $p$ surjectively onto an open subgroup of $\GL_d(\ZZ_p)$, then $\rho$ has \B. More generally, we show that if the image of inertia is open in the image of the decomposition group, the normal closure of the local image is sufficiently large in the global one, and a certain condition on the center of $\rho(G_\KK)$ is satisfied, then $\rho$ has \B. In particular, no assumption on the modularity of $\rho$ is needed, contrary to previous work of Habegger and Amoroso--Terracini. We provide several examples both in modular and non-modular cases. Our methods rely on a result of Sen comparing the ramification and Lie filtrations on the $p$-adic Lie group $\rho(G_\KK)$.
\end{abstract}

\maketitle

\section*{Introduction}
\noindent A set $\mathcal{Y}$ of algebraic numbers has the \textit{Bogomolov property}, shortened as `property \B', if there is no infinite sequence of non-torsion elements $\alpha\in \mathcal{Y}\setminus \{0\}$ whose absolute, logarithmic Weil height tends to 0.  Property \B was introduced in \cite{BombieriZannier2001} as a weakening of the Northcott property (N), that holds when the set $\mathcal Y$ contains only a finite number of elements of any bounded height. One easily sees that \B is not satisfied by $\mathcal Y=\ovl\Q$, for which the (open) Lehmer conjecture predicts a lower bound on the height of non-torsion elements as a function of the degree.
However, property \B has been established for a large class of algebraic extensions: for an overview of the results in this area, we refer the reader to the introduction of \cite{AmorosoTerracini2024}. We limit ourselves to mention that \B holds for $\KK^{ab}$, the maximal abelian extension of a number field $\KK$ (see \cite{AmorosoDvornicich2000} for the case $\KK=\QQ$ and \cite{AmorosoZannier2000} for the general case), and for the field $\QQ(E_{tors})$ obtained by adding to $\QQ$ the coordinates of the torsion points of an elliptic curve $E$ defined over $\QQ$ \cite{Habegger2013}. In the spirit of generalizing the previous results, it is natural to investigate \B for the extension $\KK(\rho)$ fixed by the kernel of a continuous representation $\rho$ of the absolute Galois group $G_\KK=\Gal(\ovl\KK/\KK)$, with $\KK$ a number field. In this case, we say that $\rho$ has \B if $\KK(\rho)$ has \B, and the results cited above then show that all characters of $G_\KK$ have \B, as well as the adelic Galois representation attached to an elliptic curve over $\QQ$. This point of view was introduced in \cite{AmorosoTerracini2024}, where the authors proved \B for certain adelic modular representations \cite[Theorems 1.4 and 1.5]{AmorosoTerracini2024}.


One of the key ingredients in the proof of property \B for an algebraic field $\LL$ in the cases we cited above is the use of \textit{metric inequalities} at a fixed prime $p$: for a given $\alpha\in \calo_\LL$, one finds a conjugate $\alpha'\not=\alpha$ of $\alpha$ such that $\lvert\alpha-\alpha'\rvert_v$ is bounded by a constant $C$, for every $v$ in a sufficiently large set $S$ of places above $p$. This is first established at a single $v$ and then extended to a suitable set of conjugates of $v$. The local task at $v$ amounts to bounding the ratio between the ramification index of $v$ and the index of the last non-trivial ramification group (in lower notation) of a finite subextension of $\LL$ containing $\alpha$. 
In the abelian case, the required bound follows from the Hasse--Arf theorem \cite[Proposition 2.3]{AmorosoZannier2010}. For fields of the form $\QQ(E_{tors})$, the proof relies on Lubin--Tate theory applied to the formal group of $E$ at a supersingular prime \cite[Lemma 3.3 (ii)]{Habegger2013}. The same argument can be applied to the Galois representation associated to a cuspidal form $f=\sum a_nq^n$ in the special case when $a_p=0$ for a prime $p$ not dividing the level (see the proof of Theorem 1.4 in \cite{AmorosoTerracini2024}).

We apply the method of metric inequalities to a class of fields that appears naturally in the theory of Galois representations, namely that of Galois extensions $\LL/\KK$, $\KK$ a number field, for which $\Gal(\LL/\KK)$ is a \textit{$p$-adic Lie group}, i.e. is equipped with a $p$-adic analytic structure compatible with the group operation. This is always the case if $\LL=\KK(\rho)$ for a continuous representation $\rho:G_\KK\to\GL_d(\Qp)$. In addition to the upper and lower ramification filtrations, the group $G\coloneqq\Gal(\LL/\KK)$ comes equipped with a Lie filtration that enjoys very good properties: for instance, the graded pieces are eventually all of the same size, determined by the dimension of the Lie group. 
After restricting to a decomposition group $G_v \subset G$ at a place $v$ of $\KK$ above $p$, the ramification and Lie filtrations are related by a very general result proved by Sen \cite{Sen1972}, answering a question of Serre \cite{Serre1967}. 
According to Sen's theorem (Theorem \ref{teo:sen}), the ramification filtration of $G_v$ in the upper numbering is equivalent to the Lie filtration, up to scaling the indices of the former by a factor equal to the ramification index of $\KK_v/\mathbb{Q}_p$. This allows us to relate the Lie filtration to the ramification filtration with lower numbering (Theorem \ref{teo:bound1}), and in turn to derive the metric inequalities for $\LL$.


Via Sen's theorem and its consequences, we are able to prove the Bogomolov property for a class of algebraic extensions $\LL/\Q$ whose Galois groups $G$ (over some intermediate number field $\KK$) are compact \(p\)-adic Lie groups; namely those that enjoy three properties: we require that $\LL/\KK$ has the \textit{central element property} (there exists an element of $G$ that acts on $p$-power roots of unity by raising them to an integer power $g>1$, see Definition \ref{def:CEP}), it is locally potentially totally ramified (an inertia subgroup at a place of $\KK$ over $p$ is open in the decomposition group, see Definition \ref{def:PTR}), and it has the \emph{normal closure property} (the normal closure of a decomposition group is sufficiently large in $\Gal(\LL/\KK)$, see \NC.

\begin{customthm}{1}[{Theorem \ref{teo:BforLie}}]\label{thm:intro1NC}
Let $\KK$ be a number field and $\LL/\KK$ be a Galois extension such that $\Gal(\LL/\KK)$ is a $p$-adic Lie group. If $\LL/\KK$ is locally potentially totally ramified at a place of $\KK$ above $p$, has the central element and the normal closure property, then it has property \B.
\end{customthm}

\noindent The assumptions of Theorem \ref{thm:intro1NC} are motivated by the following steps in the proof:
\begin{itemize}
\item[(1)] Once a metric inequality has been detected at a place \(v\) above \(p\), it becomes necessary to enlarge the constant \(C\) involved in the inequality. This is achieved by replacing \(\alpha\) with a power \(\alpha^q\) (with \(q\) a power of \(p\)); at this step, we must ensure that the metric inequality does not become trivial, i.e., \(\alpha^q - \alpha'^q \neq 0\). The central element property allows us to work with elements $\alpha$ for which this degeneracy does not occur.
\item[(2)] When passing from a local to a global setting, it is necessary to extend the metric inequality from the fixed place \(v\) to a proportion of places above \(p\) that does not decrease along the tower of finite Galois subextensions of \(\LL/\QQ\). We show that this is always possible when $\LL/\QQ$ is locally potentially totally ramified and has the normal closure property.
\end{itemize}

\noindent When $\LL/\KK$ has the stronger property of being \emph{potentially totally ramified} at a $p$-adic place $v$ (an inertia group at $v$ is open in $\Gal(\LL/\KK)$) the assumptions of \Cref{thm:intro1NC} can be simplified, giving the following:

\begin{customthm}{2}[{Theorem \ref{teo:rhohasBcenter}}]\label{thm:intro1}
Let $\KK$ be a number field and $\LL/\KK$ be a Galois extension such that $\Gal(\LL/\KK)$ is a $p$-adic Lie group. If $\LL/\KK$ is potentially totally ramified at a place of $\KK$ above $p$ and has the central element, then it has property \B.
\end{customthm}

Despite looking like a technical assumption, the central element property is necessary, as one can easily produce a counterexample if it is removed (Remark \ref{rem:counterexample}). 
However, without this property, we can still give a lower bound on the height outside of a set of ``bad'' elements of $\LL$ (see Theorem \ref{teo:condizionale}), i.e. the elements for which step (1) above fails. Unfortunately, such a set of bad elements is generally larger than one would expect based on conjectures of Rémond, and as shown in a special case by Amoroso (see \cite{Amoroso2016} and Remark \ref{rem:amoroso}).

A compact $p$-adic Lie group can always be embedded in $\GL_d(\Z_p)$ for a sufficiently large $d$ (see Lemma \ref{lem:berger}). In particular, given a continuous representation
\[ \rho \colon G_{\KK} \;\longrightarrow\; \GL_d(\mathbb{Z}_p), \]
we can deduce from Theorem \ref{thm:intro1} a result about property \B for $\LL=\KK(\rho)$. In this setting, both the central element property and potential total ramification are satisfied under a ``full'' or ``big'' inertial image condition (Assumption \eqref{assumption:main}), leading us to the following. We fix a decomposition group $G_v\subset G_\KK$ at $v$, and an inertia subgroup $I_v\subset G_v$.

\begin{customcor}{3}[Theorem \ref{teo:B-simplified}]\label{cor:intro1}
If $\rho(I_v)$ is open in $\GL_d(\mathbb{Z}_p)$, then $\KK(\rho)$ has \B.
\end{customcor}

 This result, together with the applications that we provide in the text, produces many examples of non-abelian representations of $G_\KK$, not obtained by restriction from $G_\QQ$, that enjoy property \B.

The fullness condition can essentially be checked on a decomposition group: we show that $\rho(I_v)$ is open in $\GL_d(\mathbb{Z}_p)$ if and only if $\rho(G_v)$ contains an open subgroup of $\SL_d(\ZZ_p)$ and the determinant is not trivial on the wild inertia subgroup (Proposition \ref{prop:inaspettata}).

 The novelty of these results, with respect to previous examples of Galois representations $\rho$ with property \B, is that no assumption on either the dimension or the origin of $\rho$ is made: in particular, as we explain below, many non-modular representations with property \B can be produced already in the 2-dimensional case. To our knowledge, these are the first examples when $\rho$ is not assumed to be ``geometric''.

Our next step is to produce examples of representations that satisfy the assumptions of Corollary \ref{cor:intro1}. 
In the 2-dimensional case, it turns out that the openness condition on the image of inertia is not very restrictive. An explicit classification of the possible Lie algebras of a \(p\)-adic Lie group \(G\subseteq \GL_2(\ZZ_p)\) allows one to show that $G$ is full if and only if the $G$-module \(\QQ_p^2\) is strongly absolutely irreducible (Proposition \ref{prop:liealg}). From this, one can easily characterize representations satisfying \eqref{assumption:main} (Corollary \ref{cor:bigimage}), as the absolutely irreducible representations that are not induced, and whose determinant is infinite on $I_v$. In particular, we deduce the following theorem, that is stated in the main text for general $\KK$ but we reproduce here only for $\KK=\QQ$ in order to ease the exposition:

\begin{customthm}{4}[Theorem \ref{thm:propsimplyB}]\label{thm:intro2}
Let $\rho:G_\Q\to\GL_2(\Z_p)$ be a continuous representation such that $\rho\vert_{G_p}$ is absolutely irreducible and not induced. Assume that the Hodge--Tate--Sen weights $k_1,k_2$ of $\rho\vert_{G_p}$ satisfy $k_1\ne\pm k_2$. Then $\rho$ has \B.
\end{customthm}

This criterion applies in particular when $\rho$ is the Galois $p$-adic representation associated with a modular eigenform:

\begin{customcor}{5}[Proposition \ref{prop:modhasB1}]\label{cor:intro4}
Let $k\ge 2$, $N\ge 1$ be two integers, and let $f$ be a cuspidal modular eigenform of weight $k$ and level $\Gamma_1(N)$. Let $p\nmid N$ be a prime, and $\fp$ a $p$-adic place of the Hecke field $\KK_f$ such that $f$ is $\fp$-supersingular, has $T_p$-eigenvalue $a_p\ne 0$, and the $\fp$-completion of $\KK_f$ is $\QQ_p$.
Then the $\fp$-adic Galois representation $\rho_{f,\fp}:G_\Q\to\GL_2(\KK_{f,\fp})$ attached to $f$ has \B.
\end{customcor}
This result is orthogonal to \cite[Theorem 1.4]{AmorosoTerracini2024}, which relies on the fact that $\rho_{f,\fp}\vert_{G_p}$ is induced when $a_p=0$, a case when the image of inertia cannot be full. In \Cref{ex:AV} we apply \Cref{cor:intro4} to the supersingular factors of Tate modules of abelian varieties of $\GL_2$-type.

Note that $\rho_{f,\fp}(I_p)$ is often full even if the mod $p$ reduction of $\rho_{f,\fp}\vert_{G_p}$ is reducible, as is the case whenever $p$ is \textit{$\Gamma_0(N)$-regular} in the sense of Buzzard. We refer to Section \ref{sec:regular} for explicit examples extracted from the LMFDB \cite{LMFDB}.

The assumptions of Theorem \ref{thm:intro2} are particularly pleasant, since they are all satisfied by a ``generic'' representation of $G_\Q$. We make this statement precise in Section \ref{sec:exampledef}, where we construct open dense subspaces of suitable (pseudo-)deformation spaces whose points all have property \B. 

If one would rather work with representations that are relevant in the theory of $p$-adic modular forms, the deformation-theoretic results can be specialized to the $p$-adic representations attached to non-ordinary overconvergent $p$-adic eigenforms for $\GL_{2/\Q}$ (i.e forms obtained by interpolating classical, $p$-supersingular eigenforms in the weight variable). We show that property \B holds generically along a family of such forms, that is, a non-ordinary component of the Coleman--Mazur eigencurve (Proposition \ref{prop:overconvergent}). 


We also give examples of representations of arbitrary dimension having property \B, relying on results of Serre and Chenevier (Theorem \ref{thm:serreB}, Proposition \ref{prop:chenevier}). 

In all of the applications we mentioned until this point the relevant field extension is potentially totally ramified at a $p$-adic place, so that \Cref{thm:intro1} suffices. Since it is unlikely for such a strong assumption to hold in more general situations of arithmetic interest, we also give a criterion for \B, for 2-dimensional representations, based on the weaker assumptions of \Cref{thm:intro1NC}; see \Cref{prop:2dim}. A mod $p$ ``large global image'' assumption is required, but this is known not to be very restrictive in arithmetic applications. We apply our result to the representation carried by the $p$-adic Tate module of an elliptic curve over a number field.

\begin{customthm}{6}[{\Cref{thm:ellK}}]
Let $E$ be a non-CM curve elliptic defined over a number field $\KK$. Assume that there exists a rational prime $p$ such that:
\begin{enumerate}[label=(\roman*)]
\item $\ovl\rho_{E,p}:G_\KK\rightarrow \GL_2(\FF_p)$ is surjective;
\item for a $p$-adic place $v$ of $\KK$, $\rho|_{G_v}$ is absolutely irreducible.
\end{enumerate} 
Then $\rho_{E,p}$, or equivalently $\KK(E[p^\infty])$, has \B.
\end{customthm}

The main limitation of our method, compared to the results in \cite{Habegger2013} and \cite{AmorosoTerracini2024}, is that it only applies to \(p\)-adic representations. Since representations of geometric origin live in compatible systems, one can investigate property \B for the compositum of the fields cut out by the $\ell$-adic member of the system as $\ell$ varies among the primes of a global field of definition, i.e. the field fixed by the corresponding adelic representation. In fact, an analysis of the proofs in \cite{Habegger2013, AmorosoTerracini2024} reveals that the assumption $a_p=0$ is used twice, for seemingly unrelated purposes. First, it is useful in order to understand the local representation $\rho\vert_{G_p}$ at the chosen prime $p$, and establish property \B for the $p$-adic representation. This is the context we set out to investigate in this work, and in which we have been able to obtain more general results. Second, the same assumption allows one to manage the factors of the adelic representation that are unramified at $p$, by exploiting the fact that a power of $\mathrm{Frob}_p$ is central (see \cite[Proposition 3.1]{AmorosoTerracini2024}). Although working with the ramified factor might seem like a more complex task, it is precisely this second application that we failed to deal with. Note that property \B is not preserved under field composition, so one has to find a way to prove it in a simultaneous way for the ramified and unramified-at-$p$ part of a compatible system. 

We conclude this introduction by pointing out a few aspects of our work that we would like to explore in more depth, and that will be the subject of future research.

\smallskip

\smallskip

\noindent\textit{Fields without the central element property.} The central element property is used to replace an element \(\alpha\) in the field with an element whose height is controlled by that of \(\alpha\), and which is not a $p^k$-th root of any element on a lower step of the Galois tower. This property already appeared in \cite[Proposition 4.2]{AmorosoDavidZannier2014}, \cite[Lemma 6.4]{Habegger2013}, and \cite[Assumption 3.2]{AmorosoTerracini2024}. 
It would be enlightening to connect it with the arithmetic of the field involved. Moreover, for a field not satisfying the central element property, we hope to refine our Theorem \ref{teo:condizionale} to identify significant subsets on which the height is bounded from below. This kind of question is related to a conjecture of Rémond, as reported in \cite{Amoroso2016}.

\smallskip

\noindent\textit{Refining the constants.} In this paper, we focused on the existence of a (non-explicit) lower bound for the Weil height of certain algebraic extensions. The study of the constant arising from Theorem \ref{teo:BforLie}, as well as a more in-depth study of its dependence on other quantities involved — such as the prime \(p\), the dimension of the Lie group, and the constant arising from Sen's theorem — will be object of future work.

\smallskip

\noindent\textit{Abelian varieties.} While we produce some examples of abelian varieties of $\GL_2$-type whose associated Galois representation has \B at a place of the field of multiplication, it would be desirable to extend our method to proving \B for the full $p$-adic torsion of some abelian varieties not necessarily of $\GL_2$-type. 

\subsection*{Notation}\label{sec:notation}
In the whole paper, $p$ is a fixed odd prime. 
Moreover, 
\begin{itemize}
\item algebraic extensions of $\QQ_p$ will be denoted by $F,K,E,L,\ldots$,
\item algebraic extension of $\QQ$ will be denoted by $\mathbb{F},\mathbb{K},\mathbb{E},\mathbb{L},\ldots$.
\end{itemize}
If $\alpha$ is an algebraic number, we denote by $h(\alpha)$ the absolute logarithmic Weil height of $\alpha$. We recall its definition in Section \ref{sec:B}. \\
For an arbitrary field $\kappa$, we denote by $\ovl \kappa$ an algebraic closure of $\kappa$, and by $G_\kappa=\Gal(\ovl \kappa/\kappa)$ the absolute Galois group of $\kappa$. \\
We denote by $G_p$ an absolute Galois group of $\Q_p$, and by $I_{p}\subset G_{p}$ its inertia subgroup. We write $I_p^w$ for the wild inertia subgroup of $I_p$, and $I_p^t$ for the tame quotient $I_p/I_p^w$. We fix an embedding $\ovl\Q\into\Qp$, identifying $G_{p}$ with a decomposition subgroup of $G_\Q$. \\
We denote by $\mug$ the set of roots of unity in $\ovl\Q$ and by $\mug_{p^\infty}$ the subgroup of roots of unity of order a power of $p$. For every field $\kappa$, we put 
$\mug_{p^\infty}(\kappa)=\kappa^\times\cap\mug_{p^\infty}$. We write $\vareps_p\colon G_\Q\to\Z_p^\times$ for the $p$-adic cyclotomic character, and use the convention that $\vareps_p$ has Hodge--Tate weight 1. \\
We write $\Mat_n(A)$ for the ring of $n\times n$ matrices with coefficients in a commutative ring $A$. For $n,d\in\Z_{\ge 1}$, we write $\Gamma_d(p^n)=\ker(\SL_d(\ZZ_p)\to\SL_d(\ZZ/p^n\ZZ))$ and $\wtl\Gamma_d(p^n)=\ker(\GL_d(\ZZ_p)\to\GL_d(\ZZ/p^n\ZZ))$, for the principal congruence subgroups of level $p^n$ of $\SL_d(\ZZ_p)$ and $\GL_d(\ZZ_p)$, respectively (see also Definition \ref{def:congsub}). Since principal congruence subgroups are sufficient to our purposes, we will always omit the ``principal’’. \\
The center of a group $G$ is denoted $Z(G)$.

\medskip

\setcounter{tocdepth}{1}
\tableofcontents


\section{A consequence of a theorem of Sen}\label{subsect:sen} 
For preliminaries on $p$-adic Lie and ramification groups, we refer to Appendix \ref{sect:lie}. 
Let $\QQ_p\subseteq F \subseteq L\subseteq \Qbar_p$ be a tower of fields such that $L/F$ is totally ramified and  $F/\QQ_p$ has a  finite ramification index $e = e(F/\QQ_p)$.
Assume that the Galois group $G = \Gal(L/F)$ is a $p$-adic Lie group, with $\dim(G) > 0$.
Let 
\[ \ldots\subseteq G_n\subseteq G_{n-1}\subseteq \ldots\subseteq G=\Gal(L/F) \]
be a Lie filtration on $G$. By a theorem of Sen (Theorem \ref{teo:sen}), there exists a constant $c > 0$ such that
\begin{equation}\label{eq:sen} G(ne + c) \subseteq  G_n\subseteq  G(ne - c),\end{equation}
for all $n$, with $G(r) = G$ for $r < 0$. We set:
\begin{itemize}
\item $Q_n=G/G_n$; $|Q_n|=e_n$,
\item $L_n=L^{G_n}$, so that $\Gal(L_n/F)=Q_n$,
\item $Q_n[i]=i$-th lower ramification groups of $L_n/F$,
\item $Q_n(i)=i$-th upper ramification groups of $L_n/F$, so that $Q_n(i)=G(i)G_n/G_n$ by \eqref{eq:quotients},
\item $h^n_i=|Q_n[i]|$.
\end{itemize} 
As recalled in Appendix \ref{appendix:ramification}, upper and lower ramification indices of $L_n/\QQ_p$ are related by the maps 
\[ \varphi_n=\varphi_{L_n/\QQ_p}, \quad \psi_n=\psi_{L_n/\QQ_p}, \]
where $\psi_n=\varphi_n^{-1}$ and
\[ Q_n(\varphi_n(i))=Q_n[i], \quad\quad Q_n[\psi_n(i)]=Q_n(i). \]
Recall that for $m\leq u\leq m+1$
\[\varphi_n(m)=\frac 1 {h^n_0}\left (h^n_1+\ldots + h^n_{m}+(u-m)h^n_{m+1}\right ). \]
Notice that $h^n_0=e_n$.

We recall that, for a finite Galois group $H$, any integer number $r$ such that $H[r]\supsetneq H[r + \epsilon]$ for all $\epsilon>0 $ is called a {\sl jump} in
the chain $(H[i])_i$. 

Let  $r^n_1<\ldots, <r^n_{s(n)}$ be the jumps in the chain $(Q_n[i])_i$, so that 
\[Q_n=Q_n[r_1^n]\supsetneq Q_n[r_2^n]\supsetneq \ldots\supsetneq Q_n[r_{s(n)}^n]\supsetneq \{1\}.\]

By Theorem \ref{teo:Lieuniform}, the $p$-adic Lie group $G$ admits a uniform pro-$p$ open subgroup $H$. In particular, for $m$ large enough, $G_m\subset H$, so that $[G_{m+n}:G_{m+n-1}]\leq [H_{n}:H_{n-1}]$ which is independent of $n\ge 0$. Therefore, there exists $A>0$ such that for every $n$, $[G_n:G_{n-1}]\leq p^A$.

\begin{theorem}\label{teo:bound1} 
Assume that the above conditions  are satisfied, and that $[F:\QQ_p]$ is finite. 
Then for $j_0\in\NN$, there exists a constant $C>0$ (depending on $j_0$) such that for every $n$ such that $s(n)\geq j_0$,
\begin{equation*}  \frac{r^n_{{s(n)}-j_0}} {e_n}\geq C.\end{equation*}
\end{theorem}
\begin{proof} Let $k_F$ be the residue field of $F$, and $f=[k_F:\FF_p]$. By \cite[Ch. IV, \S 2]{Serre1968}, the quotient $Q_n/Q_n[1]$ is isomorphic to a subgroup of $k_F^\times$, and for $j\geq 1$, the quotients $Q_n[j]/Q_n[j+1]$ are isomorphic to subgroups of $k_F$. It follows that
\begin{equation}\begin{split}\label{eq:pf}
& p\leq |Q_n[r^n_{s(n)}]| \leq  p^f,\\
& p\leq [Q_n[r^n_{i}]:Q_n[r^n_{i+1}]]\leq p^f   \quad \hbox{
for } i=2,\ldots, s(n)-1,
\\
& 1<[Q_n:Q_n[r^n_{2}]]\leq  p^f. 
\end{split} \end{equation}
Hence
\begin{align}\label{hn} h^n_{r^n_{s(n)-j}}\leq  p^{f(j+1)} & \hbox{ for } j=0,\ldots, s(n)-1.
\end{align}

Let $0\leq j_0\leq s(n)-2$.
Then for $j=1,\ldots ,s(n)-j_0-1$
\begin{equation}\label{eq:dis0} \begin{split} \varphi_n(r^n_{{s(n)}-j_0})& -\varphi_n(r^n_{s(n)-j_0-j}) =\frac 1 {e_n}(h^n_1+\ldots + h^n_{r^n_{{s(n)}-j_0}})-\frac 1 {e_n}(h^n_1+\ldots +h^n_{r^n_{s(n)-j_0-j}}), \quad\hbox{ by }\eqref{eq:phi},\\
& = \frac 1 {e_n}(h^n_{r^n_{s(n)-j_0-j}+1}+\ldots +h^n_{r^n_{s(n)-j_0}})\\
&= \frac 1 {e_n}(r^n_{s(n)-j_0-j+1}-r^n_{s(n)-j_0-j})h^n_{r^n_{s(n)-j_0-j+1}}+\ldots + \frac 1 {e_n}(r^n_{s(n)-j_0}-r^n_{s(n)-j_0-1})h^n_{r^n_{s(n)-j_0},}\intertext{where we used the fact that $h^n_{a}=h^n_{b}$ for $r^n_i<a,b\le r^n_{i+1}$, by definition of the jumps $r^n_i$, }
& \leq \frac 1 {e_n}(r^n_{s(n)-j_0}-r^n_{s(n)-j_0-j}) h^n_{r^n_{s(n)-j_0-j+1}}
\intertext{because $h^n_{a}\ge h^n_b$ when $a\le b$,}
& \leq \frac {r^n_{s(n)-j_0}} {e_n} p^{f(j_0+j)},\quad \hbox{by } \eqref{hn}.
\end{split}\end{equation}
Let $k_0\geq 1$ be such that there are at least $j_0+1$ jumps between $G_n$ and $G_{n-k_0}$: thanks to \eqref{eq:pf}, it is enough to choose $k_0$ such that $[G_{n-{k_0}}:G_n]=p^{k_0A} \geq p^{f(j_0+1)}$  Consider the chain
\begin{equation*}\label{chain1} G_n\subsetneq G_{n-k_0}\subseteq G(e(n-k_0)-c)
\subseteq G(e(n-k_0-1)-c)\subseteq G_{n-k_0-1-\lceil \frac {2c}e\rceil}.
\end{equation*}
Passing to the quotient by $G_n$ we get
\[\{1\}\subsetneq Q_n (e(n-k_0)-c)=Q_n[\psi_n(e(n-k_0)-c)]\]
so that by definition $r^n_{s(n)-j_0}\geq \psi_n(e(n-k_0)-c)$,  that is 
\begin{equation}\label{eq:dis1} \varphi_n(r^n_{s(n)-j_0})\geq e(n-k_0)-c.\end{equation}
Notice that $[G_{n-k_0-1-\lceil \frac {2c}e\rceil}:G_n]\leq p^{A(k_0+1+\lceil \frac {2c}e\rceil)}$, so that there cannot be more that $A(k_0+1+\lceil \frac {2c}e\rceil)$ jumps between $G_n$ and $G_{n-k_0-1-\lceil \frac {2c}e\rceil)}$. It follows that \begin{equation}\label{eq:dis2} \varphi_n(r^n_ {s(n)-A(k_0+1+\lceil \frac {2c}e\rceil)})\leq e(n-k_0-1)-c.\end{equation}
Put $j=A(k_0+1+\lceil \frac {2c}e\rceil)-j_0$. 
Then, by \eqref{eq:dis1} and \eqref{eq:dis2},
\[\varphi_n(r^n_{s(n)-j_0})-\varphi_n(r^n_{s(n)-j_0-j})\geq e\]
so that by \eqref{eq:dis0}
\begin{equation}\label{rnen2} \frac {r^n_{s(n)-j_0}} {e_n}\geq \frac e {p^{fA(k_0+1+\lceil \frac {2c}e\rceil)}}. \end{equation}
\end{proof}

For $j_0=0$, we obtain the following corollary.

\begin{corollary}\label{cor:rapplim}
The quotients $\frac {r^n_{s(n)}} {e_n}$ are bounded from below by a constant as $n\to+\infty$.
\end{corollary}

Theorem  \ref{teo:bound1} enables us to prove the existence in $G/G_n$ of a  ramification group containing $G_{n-1}/G_n$, and such that its lower index is not too small.

\begin{proposition}\label{prop:indexbound}
For every $n\geq 1$, there exists $t_n\in\NN$ such that $Q_n[t_n]$ contains $G_{n-1}/G_n$, and that $\frac{e_n}{t_n}$ is bounded from above by a constant as $n\to\infty$. 
\end{proposition}

\begin{proof}  
With the notation we introduced earlier, consider the tower of groups
\[G_n\subsetneq G_{n-1}\subseteq G(e(n-1)-c)\subseteq G_{n-1-\lceil\frac {2c} e\rceil }.\] There are at most $A(1+\lceil\frac {2c} e\rceil)$  jumps between $G_n$ and $G_{n-1-\lceil\frac {2c} e\rceil}$. Then $G(e(n-1)-c)\subseteq G(\varphi_n(r^n_{s(n)-A(1+\lceil\frac {2c} e\rceil)}))$, so that \[ G_{n-1}\subseteq G(\varphi_n(r^n_{s(n)-A(1+\lceil\frac {2c} e\rceil)})). \] Put  $t_n=r^n_{s(n)-j_0}$, with $j_0=A(1+\lceil\frac {2c} e\rceil)$. Taking quotients by $G_n$ we obtain
\[G_{n-1}/G_n \subseteq Q_n[t_n].\] Then the claim follows by applying Theorem \ref{teo:bound1}, thanks to the fact that the constant $C$ only depends on $j_0$, which is independent of $n$.
\end{proof}

We keep the notation and hypotheses introduced in the beginning of the section for the tower of fields $\QQ_p\subseteq F \subseteq L\subseteq \Qbar_p$. For $\alpha\in\Qbar_p\setminus F$ we define $$\omega_F(\alpha)=\min_{\alpha'\not=\alpha} \{|\alpha'-\alpha|_p\}$$
where $\alpha'$ varies in the set of conjugates of $\alpha$ over $F$, and $$\omega_{L/F}=\sup_{\alpha\in \calo_L\setminus \calo_F}\omega_F(\alpha).$$
Proposition \ref{prop:indexbound} has the following consequence.

\begin{corollary}\label{cor:localkey} 
If $F/\QQ_p$ is finite, then $\omega_{L/F}$ is finite. That is, for some constant $B$ the following holds: for every $\alpha\in \calo_L\setminus\calo_F$ there exists $\sigma\in \Gal(L/F)$ such that $0<|\sigma(\alpha)-\alpha|_p< B$.\end{corollary}

\begin{proof} 
Firstly notice that if $L/F $ is finite then the claim is true: in fact $\calo_L=\calo_F[\beta]$ for some $\beta$ \cite[Ch. III, \S 6, Prop. 12]{Serre1968}, so that $|\sigma(\alpha)-\alpha|_p\leq |\sigma(\beta)-\beta|_p$  for every $\sigma\in \Gal(L/F)$. Then we can take $B=\max_{\sigma\not=id} |\sigma(\beta)-\beta|_p$ and we are done.
Assume now that $L/F$ is infinite. Since $G_0$ is open in $G$, $L_0/F$ is finite and it is enough to prove the claim for $\alpha\not\in L_0$. Let $n\geq 1$ be the smallest index such that $\alpha  \in \calo_{L_n}$.
Then there exists $\sigma\in \Gal(L_n/L_{n-1})=G_{n-1}/G_n$ such that $\sigma(\alpha)\not=\alpha$. By Proposition \ref{prop:indexbound}, there exists a constant $C$, independent of $n$, and an index $t_n\geq Ce_n$ such that $\sigma\in Q_n[t_n]$; then $|\sigma(\alpha)-\alpha|_p< p^{-{\frac{t_{n+1}}{e_n}}} < p^{-C}.$
\end{proof} Corollary \ref{cor:localkey} is the key local step in proving the metric inequalities 
required to establish the Bogomolov property under our assumptions. Beyond that, a globalization procedure 
will be needed in order to show that if \(\alpha\) is global, then the same \(\sigma\) works for 
a sufficiently large set of places above \(p\).

\section{Property \B for locally potentially totally ramified $p$-adic Lie extensions}\label{sec:B}

 In this section, we prove a lower bound for the Weil heights of algebraic numbers in Galois extensions of a number field that are $p$-adic Lie and ``locally potentially totally ramified'', up to excluding an exceptional set of elements (see Theorem \ref{teo:condizionale}). If moreover the extension satisfies the ``central element property'', such a set is as small as possible (see Theorem \ref{teo:BforLie}), and we obtain the ``Bogomolov property'' that we recall just below.

We recall the definition of the (absolute, logarithmic) Weil height $h(\alpha)$ of an algebraic number $\alpha$. Given a number field $\KK$ containing $\alpha$,  
\begin{equation*} h(\alpha)=\frac 1{[\KK:\QQ]} \sum_{v} [\KK_v:\QQ_p] \log \max\{1,|\alpha|_v\},\end{equation*}
where $v$ varies over the set of places of $\KK$ and $v\mid p$. It is easily checked that such a definition depends only on $\alpha$, not on $\KK$. 

Following \cite{BombieriZannier2001}, we say that:

\begin{definition}
A set $\mathcal{Y}$ of algebraic numbers has the \textup{Bogomolov property} (property \B for short) if there is a positive constant $C$ such that $h(\alpha)\geq C$ for every $\alpha\in\mathcal{Y}\setminus(\mug\cup\{0\})$. 
\end{definition}

\noindent In our work, $\mathcal Y$ will be a set of elements of a field extension satisfying the following property.

\begin{definition}\label{def:PTR} 
Let $\KK$ be a number field and $v$ a non-Archimedean place of $\KK$.  We say that a Galois extension $\LL/\KK$ is

\begin{description}
\item[{\crtcrossreflabel{\textup{(PTR)}}[def:ptr]}] \textup{potentially totally ramified} at $v$ if an inertia subgroup $I_v$ at $v$ is open in $\Gal(\LL/\KK)$;
\item[{\crtcrossreflabel{\textup{(LPTR)}}[def:lptr]}] \textup{locally potentially totally ramified} at $v$ if a decomposition group $D_v$ at $v$ contains the inertia subgroup $I_v$ as an open subgroup.
\end{description}

\end{definition}

\noindent Since the conjugate of an open subgroup  by an element of $\Gal(\LL/\KK)$ is also open, the definition does not depend on the choice of the inertia (resp. decomposition) subgroup  at $v$.

\begin{remark} Let $\Gal(\LL/\KK)$ be a $p$-adic Lie group, and $v$ a non Archimedean place of $\KK$. Every decomposition subgroup $D_v$ at $v$ is closed in $\Gal(\LL/\KK)$, so that it is itself a $p$-adic Lie group. Then $D_v$ has a fundamental system of neighbourhoods of unity consisting of pro-$p$-subgroups, so that $I_v$ is open in $D_v$ if and only if its maximal pro-$p$-subgroup $I_v^w$ is.  
\end{remark}

Let $\GG$ be a compact $p$-adic Lie group, $\KK$ a number field, and $\LL/\KK$ a Galois extension with Galois group $\GG$.  Let $K, L$ be the closures of $\KK, \LL$, respectively, in $\Qbar_p$, determined by the fixed embedding $\iota:\Qbar\into\Qbar_p$, and let $G$ (resp. $ I$) be the corresponding decomposition (resp. inertia) subgroup of $G$. The choice of the embedding determines a place $v$ of $\KK$ over $p$. To say that $\LL/\KK$ is locally potentially totally ramified at $v$ means  that $L/K$ is potentially totally ramified, so that $I$ is open in $G$.   

Let $A$ be the dimension of $\GG$, and $B$ be the dimension of $G$. By choosing $\GG_0$ to be a powerful, finitely generated, pro-$p$ open subgroup of $\GG$ as in \Cref{thm:Liepowerful}, we see that $\GG$ has a Lie filtration
\begin{equation}\label{eq:filglob} \ldots \GG_{n+1}\subseteq \GG_n\subseteq\ldots\subseteq  \GG_0\subseteq \GG,\end{equation}
such that for every $n\ge 0$, 
$\GG_n$ is open in $\GG$, $\GG_n^p=\GG_{n+1}$ and the quotients $\GG_n/\GG_{n+1}$ are abelian, $p$-torsion and have the same size
\begin{equation}\label{eq:A}
[\GG_n:\GG_{n+1}]=p^A
\end{equation} 
for every $n$. 

Since $I$ is closed in $G$, it is a $p$-adic Lie subgroup of $\GG$: in particular, the filtration \begin{equation}\label{eq:filloc} \ldots G_{n+1}\subseteq G_n\subseteq\ldots\subseteq  G_0\subseteq I\subseteq G,\end{equation} with $G_n=\GG_n\cap G$ is a Lie filtration of $I$ (and of $G$). Up to rescaling  the indices of \eqref{eq:filglob},  we can assume that for $n\geq 0$, $G_n^p=G_{n+1}$ and the quotients $G_n/G_{n+1}$ are abelian, $p$-torsion and have the same size
\begin{equation}\label{eq:B}
[G_n:G_{n+1}]=p^B
\end{equation} 
for every $n$. 
We also put $\LL_n=\LL^{\GG_n}$, $L_n=L^{G_n}$.

We make the following \emph{Normal Closure} assumption:

\begin{description}
\item[{\crtcrossreflabel{\textup{(NC)}}[eq:NC]}]
there exists $n_0\in\NN$ such that, for every $n\geq n_0$, the normal closure of $G_{n-1}/G_{n}$ in $\GG/\GG_{n}$ is $\GG_{n-1}/\GG_{n}$.
\end{description}

\begin{remark} The normal closure of $G_{n-1}/G_{n}$ in $\GG/\GG_{n}$ is always contained $\GG_{n-1}/\GG_{n}$, because the latter is normal in $\GG/\GG_{n}$.

Condition \NC is strictly weaker than asking that, for $n\geq n_0$, the normal closure of $G_{n}$ in $\GG$ is $\GG_{n}$.
\end{remark}

\begin{remark}\label{rem:PTRimpliesNC} When $\LL/\KK$ satisfies \PTR, then for $G_n=\GG_n$ for $n\gg 0$, so that condition \NC is automatically satisfied. 
\end{remark}


\begin{theorem}\label{teo:condizionale} Assume that $\LL/\KK$ satisfies assumptions  \LPTR  and \NC. 
Then the set 
\begin{equation}\label{eq:stranadef} \mathcal{Y}=\LL_0\cup \bigcup_{n}\{\alpha\in\LL_n \ |\ \exists\sigma\in \GG_{n-1} \hbox{ such that } \sigma(\alpha)/\alpha\not\in\mug_{p^\infty}\} \end{equation}
has property \B.\end{theorem}

Notice that by Kummer theory,
$\mug_{p^\infty}$ can be replaced with $\mug_p$ in \eqref{eq:stranadef}, since $\GG_n/\GG_{n+1}$ is $p$-torsion.

\begin{proof} 
Since $\GG_0$ is open in $\GG$ and $G_0$ is open in $G$, $\LL_0/\KK$ and $L_0/K$ are finite extensions. By assumption \LPTR, up to rescaling the indices in the filtration \ref{eq:filglob}, we can assume that  $L/L_0$ is totally ramified

Let $\alpha\in\mathcal{Y}\setminus(\mug\cup\{0\})$. Define $n\geq 0$ to be the smallest integer such that $\alpha\in \LL_n$. By Northcott's theorem, there exists a constant $c_0>0$ such that $h(\alpha)\geq c_0$ if $n=0$, since $\LL_0/\KK$ is a finite extension and $\alpha\not\in\mug$. Therefore, we may assume $n\geq 1$.

By the minimality of $n$, we have $\alpha\not\in \LL_{n-1}$. 
We claim that $\alpha\in\calY$ implies the existence of $\sigma\in G_{n-1}$ and $\tau\in \GG$ such that
\[
\frac{\sigma\tau(\alpha)}{\tau(\alpha)}\not\in \mug_{p^\infty}.
\]
Otherwise, for every $\sigma\in G_{n-1},\tau\in\GG$, $\tau^{-1}\sigma\tau$ would fix a $p$-power of $\alpha$, so that, by assumption \NC, every element of  $\GG_{n-1}$ would fix a suitable $p$-power of $\alpha$, in contradiction with the hypothesis that $\alpha\in\mathcal{Y}$. Therefore, up to replacing $\alpha$ by $\tau(\alpha)$ for a suitable $\tau\in\GG$, which does not affect the height, we can find $\sigma\in G_{n-1}$ that satisfies
\begin{equation}
\label{eq:tecnica}
\sigma(\alpha)^{p^\lambda}\neq\alpha^{p^\lambda} \quad \text{for any } p\text{-power } p^\lambda.
\end{equation}

By Proposition \ref{prop:indexbound} , there exists an integer $t_n$ such that $\Gal(L_n/L_{n-1}) \subseteq \Gal(L_n/L_{0})[t_n]$ and $\frac{t_n}{e(L_n/L_0)}$ is bounded from below by a non-zero positive constant for $n\to\infty$. Since $L_0/\QQ_p$ is a fixed finite extension, there exists a constant $C'>0$ (independent of $n$) such that
\begin{equation*}\label{eq:seninproof}  \frac{t_n+1}{e(L_n/\QQ_p)}\geq C'. \end{equation*}
Since $\sigma\in \Gal(L_n/L_{0})[t_n]\subseteq \Gal(L_n/K)[t_n]$ 
(cf. \cite[Proposition 10.3]{Neukirch1999}), for any $\gamma\in\calo_ {\LL_n}$ we have
$$
\vert\sigma(\gamma)-\gamma\vert_{v}\leq p^{-(t_n+1)/{ e(L_n/\QQ_p)}}\leq p^{-C'}.
$$
where $v$ is the place of $\LL_n$ corresponding to the embedding $\LL_n\hookrightarrow L_n$. 
By~\cite[Lemma 2.1]{AmorosoDavidZannier2014}, there exists a positive integer $\lambda$ which is explicitly bounded in terms of $p$,  and $C'$, such that
\begin{equation}
\label{metric}
\vert\sigma(\gamma^{p^\lambda})-\gamma^{p^\lambda}\vert_v\leq p^{-{p^{T}}},
\end{equation}
where $p^T\geq 
 \frac {p^A\cdot [\KK:\QQ]}{[K:\QQ_p]}$ and $A$ is the integer from \eqref{eq:A}. 
We consider the action of $\Gal(\LL_n/\KK)$ on itself by conjugation and its stabilizer $$C(\sigma)=\{\nu\in\Gal(\LL_n/\KK)\;\vert\; \nu\sigma\nu^{-1}=\sigma\}.$$ Let $\nu\in C(\sigma)$. We use~\eqref{metric} with $\nu^{-1}\gamma$ instead of $\gamma$. Since $\sigma\nu^{-1}=\nu^{-1}\sigma$ and $\vert\nu^{-1}(\star)\vert_w=\vert\star\vert_{\nu(v)}$, we get 
$$
\vert \sigma(\gamma^{p^\lambda})-\gamma^{p^\lambda}\vert_w\leq p^{-p^{T}}
$$
with $w=\nu(v)$ and thus for any $w$ in the orbit $S$ of $v$ under the the action of $C(\sigma)$ on the places of $\LL_n$. 
Using~\cite[Lemma 1]{AmorosoDvornicich2000} as in the proof of~\cite[Lemma 2.1]{AmorosoDavidZannier2014} we find
$$
\vert\sigma(\alpha^{p^\lambda})-\alpha^{p^\lambda}\vert_w
\leq c(w)\max(1,\vert\sigma(\alpha)\vert_w)^{p^\lambda}\max(1,\vert\alpha\vert_w)^{p^\lambda},\; \forall w\in S
$$
with $c(w)=p^{-p^{T}}$. This last inequality also holds for an arbitrary place $w$ of $\LL_n$ not in $S$, with 
$$
c(w)=
\begin{cases}
1,& \hbox{if } w\nmid\infty;\;\\
2,& \hbox{if } w\mid \infty.
\end{cases}
$$
Let $d_n= [\LL_n:\Q]$ and $d_{n,w}=[(\LL_n)_w:\QQ_{\lvert w\rvert}]$ for every place $w$ of $\LL_n$ above a rational place $\lvert w\rvert$. Using the Product Formula on $\sigma(\alpha)^{p^\lambda}-\alpha^{p^\lambda}$, which is $\neq0$ by~\eqref{eq:tecnica}, as in the proof of~\cite[Lemma 2.1]{AmorosoDavidZannier2014} we get: 
\begin{align*}
0 & = \sum_w\frac{d_{n,w}}{d_n}\log\vert \sigma(\alpha^{p^\lambda})-\alpha^{p^\lambda}\vert_w\\
 & \leq \sum_w\frac{d_{n,w}}{d_n}\left(\log c(w)
 +p^\lambda\log\max\{1,\vert\sigma(\alpha)\vert_w\}+p^\lambda\log\max\{1,\vert\alpha\vert_w\}\right)\\
& = \left(\sum_{w\mid\infty}\frac{d_{n,w}}{d_n}\right)\log 2
-\left(\sum_{w\in S}\frac{d_{n,w}}{d_n}\right)p^{T}\log p + 2p^\lambda h(\alpha).
\end{align*}
Since $\sum_{w\mid\infty}\frac{d_{n,w}}{d_n}=1$ and since $d_{n,w}=d_{n,v}$ for any place  $w\in S$, we get
\begin{equation}
\label{first-bound}
2p^\lambda h(\alpha)\geq \frac{d_{n,v}}{d_n}\vert S\vert\cdot p^{T}\log p-\log 2.
\end{equation}
To conclude, we need a lower bound for $S$. Note that $\sigma\in \Gal(\LL_n/\LL_{n-1}) \lhd  \Gal(\LL_n/\KK)$, thus the orbit $O$ of $\sigma$ by conjugation is contained in $\Gal(\LL_n/\LL_{n-1})$. Thus
$$
\vert C(\sigma)\vert =\frac{[\LL_n:\KK]}{\vert O\vert} \geq \frac{[\LL_n:\KK]}{[\LL_n:\LL_{n-1}]}\geq \frac {[\LL_n:\KK]}{p^{A}}.
$$
The stabilizer of $v$ under the action of $\Gal(\LL_n/\KK)$ is by definition the decomposition subgroup of $\Gal(\LL_n/\KK)$ at $v$, which is isomorphic to the local Galois group $\Gal(L_n/K)$; its cardinality is $[L_n:K]$.   Therefore,
$$
\vert S\vert 
=\frac{\vert C(\sigma)\vert}{\vert \mathrm{Stab}(v)\cap C(\sigma)\vert}
\geq \frac{[\LL_n:\KK]}{p^A[L_n:K]}=\frac{[K:\QQ_p]}{p^A[\KK:\QQ]}\frac {d_n}{d_{n,v}}\geq \frac 1 {p^T}\frac {d_n}{d_{n,v}}.
$$
Finally, from~\eqref{first-bound} we obtain 
\begin{align}\label{eq:lambda}
 h(\alpha)\geq \frac{\log(p/2)}{2p^\lambda}>0. &
\hfill \qed   
\end{align}
\phantom\qedhere
\end{proof}

\noindent Let $\LL/\KK$ be a Galois extension with Galois group $\GG$.
We introduce the following terminology.

\begin{definition}\label{def:CEP} We say that $\LL$ has the \emph{central element property} if  there exists $\tau \in Z(\GG)$,  such that \begin{equation}\label{eq:H1}\tag{CE}\hbox{for every  } \zeta\in\mug_{p^\infty}(\LL), \quad\tau(\zeta)=\zeta^g\quad\hbox{with } g\in\ZZ, g>1.\end{equation}
\end{definition}

\begin{proposition}\label{prop:CEP}
The extension $\LL/\KK$ has \CE if and only if one of the following properties is satisfied:
\begin{itemize}
 \item[i)] $|\mug_{p^\infty}(\LL)|<\infty$;
 \item[ii)] $\mug_{p^\infty}(\LL)=\mug_{p^\infty}$ and $\varepsilon_p(Z(\GG))^{p-1}\not=\{1\}$ in $\ZZ_p^\times$.
 \end{itemize}
\end{proposition}

\noindent In case $ii)$, the restriction of the cyclotomic character to $G_\KK$ factors through $\GG=\Gal(\LL/\KK)$, so that $\varepsilon_p(Z(G))$ is well-defined.

\begin{proof} Assume that $\LL/\KK$ has \CE, and that $\mug_{p^\infty}(\LL)$ is infinite. If $\varepsilon_p(Z(G))\subseteq \mug_{p-1}$ then an element $\tau$ satisfying \eqref{eq:H1} cannot exist, because $\ZZ\cap\mug_{p-1}=\{\pm 1\}$. Conversely: if $\mug_{p^\infty}(\LL)=\mug_{p^m}$, then $\tau=id$ works with $g=p^m+1$. On the other hand, if $\mug_{p^\infty}(\LL)=\mug_{p^\infty}$ and $ii)$ holds, then $\vareps_p(Z(G))^{p-1}$ is a nontrivial subgroup of $ 1+p\ZZ_p$, so that $\vareps_p(Z(G))\cap (1+p\ZZ_p)\not=\{1\}$. Then the latter is a nontrivial closed subgroup of $1+p\ZZ_p\simeq \ZZ_p$, so that it is also open. Since the image of $1+p\NN$ in $1+p\ZZ_p$ is dense, the claim follows.
\end{proof}

\begin{theorem}\label{teo:BforLie} Let $\KK$ be a number field,  $\GG$ a compact $p$-adic Lie group and $\LL/\KK$ a Galois  extension with Galois group $\GG$. Assume that $\LL/\KK$ satisfies condition \CE and conditions \LPTR, \NC at some place $v$ above $p$.  
 Then $\LL$ has property \B.
\end{theorem}

\noindent Note that \LPTR and \NC depend on the choice of a $p$-adic place $v$ of $\KK$, while \CE only depends on the prime $p$.


\begin{proof}
    Let $\tau \in Z(\GG)$ be a central element satisfying \eqref{eq:H1}.
Let $\alpha\in \LL^\times\backslash\mug$. We define  
\begin{equation}
\label{eq:beta}
\beta=\frac{\tau(\alpha)}{\alpha^g}.
\end{equation}
We remark that $\beta\not\in \mug$. Otherwise $gh(\alpha)=h(\alpha^g)=h(\tau(\alpha))=h(\alpha)$ which fails since $\alpha$ is not a root of unity and since $g>1$. Moreover, $h(\beta)\leq(g+1)h(\alpha)$. Thus it is enough to uniformly bound from below the height of $\beta$.\\

\begin{fact*}
For every $\sigma\in\GG$, if $\sigma(\beta)\not=\beta$ then $\tfrac{\sigma(\beta)}{\beta}\not\in \mug_{p^\infty}$.
\end{fact*}
\begin{proof}
Let us assume by contradiction that $\tfrac{\sigma(\beta)}{\beta}\in \mug_{p^\infty}$. 
By~\eqref{eq:beta} and since $\tau$ is central,  
$$
\frac{\sigma(\beta)}{\beta}
=\frac{\sigma\tau(\alpha)}{\tau(\alpha)}
\left(\frac{\sigma(\alpha^g)}{\alpha^g}\right)^{-1}
=\frac{\tau(\eta)}{\eta^g}.
$$
with $\eta=\sigma(\alpha)/\alpha$. Hence 
\begin{equation*}
gh(\eta)=h(\eta^g)=h(\tau(\eta))=h(\eta),
\end{equation*} 
which implies $h(\eta)=0$ since $g>1$. Thus $\eta$ is a root of unity. Write $\eta$ as $\eta=\eta_1\eta_2$ with $\eta_1\in \mug_{p^\infty}$ and with $\eta_2$ of order not divisible by $p$. Hence $\tfrac{\tau(\eta_1)}{\eta_1^g}=1$ by~\eqref{eq:H1}. Thus $\tfrac{\sigma(\beta)}{\beta}=\tfrac{\tau(\eta_2)}{\eta_2^g}$ has order not divisible by $p$. But $\tfrac{\sigma(\beta)}{\beta}\in \mug_{p^\infty}$ and $\tfrac{\sigma(\beta)}{\beta}\neq1$ by hypothesis, a contradiction.\phantom\qedhere\end{proof}

Let $n$ be the smallest index such that $\beta\in \LL_n$; as in the proof of Theorem \ref{teo:condizionale}, we can assume $n>1$. Since $\beta\not\in\LL_{n-1}$, there exists $\sigma\in \GG_{n-1}$ such that $\sigma(\beta)\not=\beta$. Then, by the above Fact, $\beta$ belongs to the set $\mathcal{Y}$ defined by \eqref{eq:stranadef}. Then  we can apply Theorem \ref{teo:condizionale} to deduce that $h(\beta)>C$ for a suitable constant $C>0$.
\end{proof}

\begin{remark}   The proofs of Theorems  \ref{teo:condizionale} and \ref{teo:BforLie}  closely follow that of Proposition 3.4 in \cite{AmorosoTerracini2024}, given the existence of a Lie filtration. However, Assumption 3.3 in \textit{loc. cit.} does not perfectly suit our situation, and we had to do some slight modifications by invoking our Theorem \ref{teo:bound1}.
\end{remark}

Thanks to Remark \ref{rem:PTRimpliesNC}, Theorem \ref{teo:BforLie} has the following  notable consequence:

\begin{corollary}\label{cor:PTRandCEimlyB} Let $\KK$ be a number field,  $\GG$ a compact $p$-adic Lie group and $\LL/\KK$ a Galois  extension with Galois group $\GG$. Assume that $\LL/\KK$ satisfies condition \CE and condition \PTR at some place $v$ of $\KK$ above $p$.  
 Then $\LL$ has property \B.
 \end{corollary}

\begin{remark}\label{rem:counterexample} 
Without assuming \CE, property \B is not ensured. Consider, for example, $\LL=\QQ(\mu_{p^\infty},b^{\frac 1 {p^n}},n\in\NN)$, with $b\in\N$, $b\ge 2$. The extension $\LL/\QQ$ is totally ramified  (see for example \cite{Viviani2004}). By Kummer theory, $\Gal(\LL/\QQ)\simeq \ZZ_p^\times\ltimes \ZZ_p$; the latter is a closed subgroup of $\GL_2(\ZZ_p)$, hence a $p$-adic Lie group; however, the center is trivial. Correspondingly, property \B clearly fails.
\end{remark}

\begin{remark}\label{rem:amoroso}
With the notation of Remark \ref{rem:counterexample}, let $\Gamma\subset\LL^\times$ be the multiplicative subgroup generated by $\mu_{p^n}$ and $b^{1/p^n}$ for varying $n$ and $\Gamma^{\mathrm{div}}$ the subgroup of $\alpha\in\LL^\times$ such that $\alpha^k\in\Gamma$ for some $k\in\NN$. If $p\nmid b$ and $p^2\nmid b^{p-1}-1$, Amoroso proves that the heights of the elements of $\LL\setminus\Gamma^{\mathrm{div}}$ are bounded from below \cite[Theorem 3.3]{Amoroso2016}. As expected, the group $\Gamma^{\mathrm{div}}$ has empty intersection with the set $\mathcal{Y}$ defined in \eqref{eq:stranadef}, but unfortunately we cannot recover Amoroso’s result from our Theorem \ref{teo:condizionale}, since the complement of $\mathcal Y$ contains many more elements than those in $\Gamma^{\mathrm{div}}$: if $\gamma\in\Gamma^{\mathrm{div}}\cap\LL_n\setminus\LL_{n-1}$, no element of the set $\LL_{n-1}\gamma$ belongs to $\mathcal{Y}$. 
In order to recover Amoroso’s result, a subtler study of the behaviour of the Weil height on the complement of $\mathcal{Y} $ in $\LL^\times$ is needed. 
\end{remark}

\section{Application to Galois representations}

Let $\KK$ be a number field and 
$\rho:G_\KK\to\GL_d(\ZZ_p)$ a continuous representation of arbitrary dimension $d$. 
Let $\KK(\rho):=\Qbar^{\ker(\rho)}$ be the field cut out by $\rho$. Following \cite[Definition 1.1]{AmorosoTerracini2024}, we shall say that $\rho$ has property \B if $\KK(\rho)$ does. 

Recall that we fixed an embedding $\Qbar\hookrightarrow \Qbar_p$. This determines a $p$-adic place $v$ of $\KK$, and allows us to identify $G_{\KK_v}=\Gal(\ovl{\QQ}_p/\KK_v)$ with a decomposition subgroup $G_v$ of $G_\KK$ at $v$. We denote by $I_v$ the inertia subgroup of $G_{v}$. 


 We say that $\rho$ satisfies conditions \PTR, \LPTR, \NC, \CE if $\KK(\rho)$
does. 

Since $\rho$ induces a continuous injective map from a compact space to a Hausdorff one, it establishes a homeomorphism between $\Gal(\KK(\rho)/\KK)$ and $im(\rho)\subseteq \GL_d(\ZZ_p)$. Therefore $\Gal(\KK(\rho)/\KK)$ is a $p$-adic Lie group. Conversely, Lemma \ref{lem:berger} ensures that if $\LL/\KK$ is an algebraic Galois extension whose Galois group is  a $p$-adic Lie group, then $\LL=\KK(\rho)$ for a continuous representation $\rho$.\\
Then Theorem \ref{teo:BforLie} and Corollary \ref{cor:PTRandCEimlyB} can be restated in terms of representations as follows:

\begin{theorem}\label{teo:rhohasBcenter} Let $\rho:G_\KK\rightarrow \GL_d(\ZZ_p)$ be a continuous representation satisfying $\CE$. Assume that one of the following hypotheses are satisfied:
\begin{itemize}
    \item [a)] For a place $v$ of $K$ above $p$, $\rho$ satisfies \LPTR and \NC.
    \item[b)] For a place $v$ of $K$ above $p$, $\rho$ satisfies \PTR.
\end{itemize}  Then $\rho$ has property \B.
\end{theorem}

 A very special case is that of a potentially totally ramified representation whose image is open in $\GL_d(\ZZ_p)$, that is such that the image of inertia is open in $\GL_d(\ZZ_p)$.
\noindent The following proposition shows that \CE automatically holds in the special case of a potentially totally ramified representation whose image is open in $\GL_d(\ZZ_p)$, that is such that the image of inertia is open in $\GL_d(\ZZ_p)$. 

\begin{proposition}\label{prop:integer} 
Assume that for some place $v$ of $\KK$ above $p$ \begin{equation}\label{assumption:main}\tag{BLI} \rho(I_v) \hbox{ is open in }\GL_d(\ZZ_p).\end{equation} 
Then the extension $\KK(\rho)/\KK$ has \CE. 
\end{proposition}


\begin{proof}
We want to show that there exists $\tau\in I_{v}$ such that $\rho(\tau)$ is a scalar matrix and $\vareps_p(\tau)$ is an integer greater than 1.\\ 
Let $\Sigma$ be the set of embeddings $\KK_v\into\Qp$. We write $\langle\sigma\rangle:\cO_{\KK_v}^\times\into\ZZ_p^\times$ for the homomorphism that coincides with $\sigma\in\Sigma$ on the pro-$p$ factor of $\cO_{\KK_v}^\times$ and is trivial on the prime-to-$p$ factor. 
A character $\cO_{\KK_v}^\times\to\ovl\ZZ_p^\times$ is of the form $\chi\prod_{\sigma\in\Sigma}\langle\sigma\rangle^{a_\sigma}$, with $\chi$ a finite order character and $a_\sigma\in\ZZ_p$. In particular, it takes values in $\ZZ_p^\times$ if and only if $\chi$ is $\ZZ_p$-valued and $a_\sigma$ is independent of $\sigma$, i.e. if it is of the form $\chi\cdot\langle\mathrm{Norm}_{\KK_v/\Q_p}^a\rangle$ for some $a\in\ZZ_p$, where $\langle\cdot\rangle$ has the same meaning as before. By applying this to $\det\rho\vert_{I_v}\ccirc\mathrm{rec}_v$ and to $\vareps_p\vert_{I_v}\ccirc\mathrm{rec}_v$, with $\mathrm{rec}_v:\cO_{\KK_v}^\times\to I_v$ a local reciprocity map, we obtain that there exist $a_1,a_2\in\ZZ_p$ such that $\det\rho\vert_{I_v}^{a_1}=\chi\vareps_p\vert_{I_v}^{a_2}$.
By hypothesis \eqref{assumption:main},
$a_1, a_2$ are both nonzero. Let  $\mathcal{S}$ be the group of scalar matrices in $\rho(I_v)$; again by \eqref{assumption:main}, $\det(\mathcal{S})$ is open in $\ZZ_p^\times$, hence  $\varepsilon_p(I_v\cap \rho^{-1}(\mathcal{S}))$ is also open in $\ZZ_p^\times$, so that its intersection with $\NN$ is infinite.  
\end{proof}

\begin{theorem}\label{teo:B-simplified}\label{thm:rhohasB}
Under Assumption \eqref{assumption:main}, $\rho$ has property \B.
\end{theorem}

\begin{proof} Since $\rho$ also satisfies \PTR, we deduce that $\rho$ has property \B from Theorem \ref{teo:rhohasBcenter} and Proposition \ref{prop:integer}.
\end{proof}

\begin{remark}\label{rem:treuno}\mbox{ }
\begin{itemize} \item[\textit{a)}] Assumption \eqref{assumption:main} is equivalent to requiring  the image of the wild inertia to be open in $\GL_d(\ZZ_p)$. Indeed, since $\tilde\Gamma_d(p^n)$ is a pro-$p$-group, it is contained in $\rho(I_v)$ if and only it lies into the maximal pro-$p$ subgroup $\rho(I_v^w)\subset\rho(I_v)$. 
\item[\textit{b)}] By Lemma \ref{lemma:SLGL}, Assumption \eqref{assumption:main} holds if and only if $\rho(I_{v})\cap\SL_d(\ZZ_p)$ is open in $\SL_d(\Z_p)$ and $\det\rho(I_v)$ is open in $\Z_p^\times$. Moreover, by Remark \ref{rem:opendet}, $\det\rho(I_v)$ is open in $\Z_p^\times$ if and only if it has infinite order, that is if and only if $\det\rho(I_v^w)\not=\{1\}$.

\end{itemize}
\end{remark}




\section{General examples}

In this section, we give some simple conditions under which a continuous representation $\rho: G_\KK\to\GL_d(\Z_p)$, $\KK$ a number field, satisfies the hypothesis of Theorem \ref{teo:rhohasBcenter} so that it enjoys property \B. 

In the following, $v$ is the $p$-adic place of $\KK$ determined by the fixed embedding $\QQ\into\Qp$. We keep the notations $I_v^w\subseteq I_v\subseteq G_v\subseteq G_\KK$ from the previous sections.

\subsection{Fullness and twists}\label{sec:fullnessandtwists}
Let $\rho:G_\KK\to\GL_d(\ZZ_p)$ be a continuous representation,
Following \cite[Definition 3.1]{ContiLangMedved2023}, we shall say that $\rho$ is \textit{inertia-full} if $\rho(I_v)$ (or equivalently $\rho(I_v^w)$) contains an open subgroup of $\SL_d(\ZZ_p)$. \\
We generalize \cite[Corollary 3.12]{ContiLangMedved2023} to arbitrary dimension, obtaining the following:

\begin{proposition}\label{prop:twistfull}
Let $\rho:G_\KK\to\GL_d(\ZZ_p)$ be an inertia-full  representation, and $\chi:G_\KK\to \ZZ_p^\times$ be a continuous character. Then the twist $\chi\otimes\rho$ is inertia-full.
\end{proposition}

\begin{proof}
Assume that $\Gamma_d(p^n)\subseteq \rho(I_v^w)$, and write $G=\rho(I_v^w)$, $G_0=\rho(\ker\chi\cap I_v^w)$. Since $G/G_0$ is abelian, $G_0$ must contain the derived group $(G,G)$. In particular it contains $(\Gamma_d(p^n),\Gamma_d(p^n))$. But the latter  is equal to $\Gamma_d(p^{2n})$ by \cite[Corollary 3.10]{ContiLangMedved2023}. Therefore $\Gamma_d(p^{2n})\subseteq \rho(\ker\chi\cap I_v^w)\subseteq (\chi\otimes\rho)(I_v^w)$.
\end{proof}

\begin{proposition}\label{prop:twistfullcons}
Let $\rho:G_\KK\to\GL_d(\ZZ_p)$ be an inertia-full representation, and let $\chi:G_\KK\to\ZZ_p^\times$ be a continuous character such that $(\chi^d\det\rho)(I_v^w)\not=\{1\}$.
Then $\chi\otimes\rho$ satisfies \eqref{assumption:main}. Therefore $\chi\otimes\rho$ has property \B.
\end{proposition}

\begin{proof} By Proposition \ref{prop:twistfull}, $\chi\otimes\rho$ is an inertia-full representation. Moreover its determinant is $\chi^d\det\rho$, so that $\chi\otimes\rho$ satisfies \eqref{assumption:main} by Remark \ref{rem:treuno} $b)$. The last claim follows from Theorem \ref{teo:B-simplified}.
\end{proof}

\begin{remark} By Proposition \ref{prop:twistfullcons}, Condition \eqref{assumption:main} is invariant by twist by any continuous
character $\chi$ such that $(\chi^d\cdot\det(\rho))(I_v^w)\not=\{1\}$. 
\end{remark}

\noindent On the other hand, property \B holds for an inertia-full representation whose determinant has finite order.

 \begin{proposition}\label{prop:detfiniteorder}
 Let $\rho:G_\KK\to\GL_d(\ZZ_p)$ be a inertia-full representation, and assume that $\det\rho$ has finite order. Then $\rho$ satisfies \PTR and \CE. Therefore, it has property \B.
 \end{proposition}
 
 \begin{proof}
Set $G=im(\rho)\cap\SL_d(\ZZ_p)$.  Since $\rho$ is inertia-full, $G$ is open in $\SL_d(\ZZ_p)$. By hypothesis, $[im (\rho): G]=\lvert\det\rho\rvert$ is finite. It follows that $\rho(I_v)$ is open in $im (\rho)$, so that $\rho$ satisfies \PTR.\\
In order to prove that $\rho$  has \CE, we show that $\mu_{p^\infty}(\QQ(\rho))$ is finite and appeal to Proposition \ref{prop:CEP}. If not, there would be a map $\theta: im(\rho)\to \ZZ_p^\times$ with open image; since $G$ has finite index in $im(\rho)$, it would map by $\theta$ onto an infinite subset of $\ZZ_p^\times$, so that the kernel $H$ of this restriction would be a non-open normal subgroup of $G$. By Proposition \ref{prop:riehm} $a)$, $H$ should consist of scalar matrices. But scalar matrices in $\SL_d(\ZZ_p)$ are a finite group of order prime to $p$, so that $H\cap\rho(I_v^w)$ would be trivial, and $\theta$ would induce an injective map $\rho(I_v^w)\cap\SL_d(\ZZ_p)\to \ZZ_p^\times$. This would imply that $\rho(I_v^w)\cap\SL_d(\ZZ_p)$ is abelian, contradicting the inertia-fullness of $\rho$.\\
Then, $\rho$ has property \B by Theorem \ref{teo:rhohasBcenter}. 
\end{proof}

\subsubsection{Constant-determinant representations}\label{sec:constdet} 
If $\rho$ is inertia-full and $\det\rho(G_\KK)$ is infinite, but $\det\rho(I_v)$ is finite, then $\rho$ is not potentially totally ramified. However, we can still twist it by a suitable character to obtain a representation with property \B, as follows.  

 Let $v_p(d)$ denote the $p$-adic valuation of $d$. Let $G$ be a closed subgroup of $\GL_d(\Z_p)$, and $\det_G:G\to\Z_p^\times$ its determinant. We can write $\det_G: G\to\Z_p^\times$ as the product of a character $\chi_0:G\to(\Z/p\Z)^\times$ (the Teichm\"uller lift of the mod $p$ reduction of $\det_G$), and a character $\chi_1:G\to 1+p\Z_p$. Assume that $\chi_1(G)\subseteq 1+p^{v_p(d)}\ZZ_p$ (this would be certainly true if either $p\nmid d$, or $G\subset\Gamma_d(p^{v_p(d)})$). In this case, we define a character $\chi_G:G\to\Z_p^\times$ with the property that $\chi_G^d\det_G$ is a finite order character:  indeed, the Taylor expansion of the $d$-th root converges on $1+p^{v_p(d)}\Z_p$ and gives a character $\chi_G$ such that $\chi_G^d=\chi_1^{-1}$, so that $\chi_G^d\det_G=\chi_0$.
We define the \textit{constant-determinant twist} of $G$ as the subgroup
\[ G_0=\{g\cdot\chi(g)\Id_d\,\vert\,g\in G\}\subset\GL_d(\Z_p). \]
Then $\det_{G_0}=\chi_0$ has finite order. The above is an analogue of \cite[Definition 2.2]{ContiLangMedved2023}, which only deals with the 2-dimensional case (the terminology ``constant-determinant'' was introduced by Bella\"iche, and refers to the fact that $\chi_0$ can be considered as a ``constant'' deformation of its mod $p$ reduction).
There might be multiple choices for the constant-determinant twist of $G$, but this ambiguity is irrelevant in the following.\\ 
If $\rho:\Pi\to\GL_d(\Z_p)$ is a representation of a profinite group $\Pi$, and either $p\nmid d$ or $\det\rho(\Pi)\subseteq (\ZZ/p\ZZ)^\times \times (1+p^{v_p(d)}\ZZ_p)$, then we define the constant-determinant twist of $\rho$ as $\rho_0=\rho\otimes\chi_{\im\rho}$, where $\chi_{\im\rho}$ is the character $G\coloneqq\im\rho\to\Z_p^\times$ constructed above, seen as a character of $\Pi$. Clearly, there exists an open subgroup $\Pi_0$ of $\Pi$ such that $\rho_0(\Pi_0)\subset\SL_d(\Z_p)$.


\begin{proposition} Assume that either $p\nmid d$ or $\det\rho(\Pi)\subseteq (\ZZ/p\ZZ)^\times \times (1+p^{v_p(d)}\ZZ_p)$, and that $\rho:G_\KK\to \GL_d(\ZZ_p)$ is inertia-full. Then its 
constant-determinant twist $\rho_0$ is satisfies \PTR and \CE. In particular, it has property \B.
\end{proposition}
\begin{proof}
By Proposition \ref{prop:twistfull}, $\rho_0$ is inertia-full, and by construction $\det\rho_0$ has finite order. Therefore, Proposition \ref{prop:detfiniteorder} is applicable. 
\end{proof}

\subsection{Some criteria in dimension 2}\label{sec:propsrhop}

Until further notice, we restrict ourselves to 2-dimensional representations. Let $\rho_v\colon G_v\to\GL_2(\Z_p)$ be a continuous representation with modulo $p$ reduction $\ovl\rho_v$. In the next propositions, we give various different assumptions on $\rho_v$ that guarantee that the image of $I_v$ is sufficiently large.

    \begin{prop}\label{prop:rhop1}
    Assume that:
    \begin{enumerate}[label=(\roman*)]
    \item $\rho_v$ is absolutely irreducible;
    \item if $\ovl\rho_v$ is reducible, then the projective image of $\ovl\rho_v^\sms$ has order different from 2;
    \item if $\ovl\rho_v$ is irreducible, then $\rho_v$ is not induced from a character, and not a twist of a representation with finite image;
    \item $\det\rho_v(I_v)$ is infinite.
    \end{enumerate}
    Then $\rho_v(I_v)$ is open in $\GL_2(\Z_p)$.
    \end{prop}

\begin{proof}
If $\ovl\rho_v$ is reducible and its projective image has order different from 2, then $\rho_v(G_p)$ contains an open subgroup of $\SL_2(\Z_p)$ by Lemma \ref{lemma:noind}. The same holds true if $\ovl\rho_v$ is irreducible,  by assumption \textit{(iii)} and Corollary \ref{cor:bigimage}. In both cases, Proposition \ref{prop:inaspettata} \textit{(ii)}   implies that $\rho_v(I_v)$ is open in $\GL_2(\Z_p)$, by assumption $(iv)$.
\end{proof}

\subsubsection{A condition on Hodge--Tate--Sen weights} Recall that Sen theory associates with $\rho_v$ and with an embedding $\sigma:K_v\into\Qp$ a \textit{Sen operator} $\Phi_v\in\GL_2(\wh{\ovl\Q}_p)$ \cite{SenHodgeTate}. The \textit{$\sigma$-Hodge--Tate--Sen weights} of $\rho_p$ are defined to be the eigenvalues $k_{\sigma,1}, k_{\sigma,2}$ of $\Phi_v$, and coincide with the $\sigma$-Hodge--Tate weights of $\rho_v$ if $\rho_v$ is Hodge--Tate.

\begin{prop}\label{prop:rhop3}
Assume that:
\begin{enumerate}[label=(\roman*)]
\item $\rho_v$ is absolutely irreducible and not induced (i.e. strongly absolutely irreducible);
\item there exists $\sigma:K_v\into\Qp$ such that $k_{\sigma,1}\ne k_{\sigma,2}$, and $\tau:K_v\into\Qp$ such that $k_{\tau,1}\ne -k_{\tau,2}$.
\end{enumerate}
Then $\rho_v(I_v)$ contains an open subgroup of $\GL_2(\Z_p)$.
\end{prop}

\begin{proof}
Since $k_{\sigma,1}\ne k_{\sigma,2}$ and representations of finite image have all Hodge--Tate--Sen weights equal to 0, $\rho_v$ is not the twist of a representation of finite image. 
If $\ovl\rho_v$ is irreducible, then $\rho_v$ is not induced by assumption. Therefore, $\rho_v(G_v)$ contains an open subgroup of $\GL_2(\Z_p)$ by Corollary \ref{cor:bigimage}. 
Moreover, $\det\rho$ is a character of $\tau$-Hodge--Tate--Sen weight $k_{\tau,1}+k_{\tau,2}\ne 0$, hence its restriction to $I_v$ has infinite order. 
The conclusion then follows from Proposition \ref{prop:inaspettata} $(ii)$.
\end{proof}


The discussion in this section up to here is summarized by the following immediate consequence of Theorem \ref{thm:rhohasB}.

\begin{thm}\label{thm:propsimplyB}
Let $\rho\colon G_\KK\to\GL_2(\Z_p)$ be a continuous representation such that $\rho_v\coloneqq\rho\vert_{G_v}$ satisfies the assumptions of either of Propositions \ref{prop:rhop1}, \ref{prop:rhop3}. Then $\rho$ has property \B.
\end{thm}

\subsubsection{A criterion of Serre} Assume that $\KK=\QQ$, so that $v=p$ and $\rho_p$ is now a representation $G_p\to\GL_2(\ZZ_p)$. We give a criterion of Serre \cite{Serre1972} for reading some of our recurring assumptions on $\ovl\rho_p$ from its restriction to $I_p$. 

Let $\ovl\rho_p\colon G_p\to\GL_2(\F_p)$ be a continuous representation, and let $\ovl\rho_p^\sms$ be its semisimplification. We describe the restriction $\ovl\rho_p^\sms\vert_{I_p}$, following \cite[Section 2]{Serre1987}. We point out a possible cause of confusion: we keep our notation for which $I_p$ is the inertia subgroup of a fixed decomposition group $G_p\subset G_\Q$, while in \textit{loc. cit.} $I$ is the full inertia subgroup and $I_p\subset I$ the wild inertia subgroup. 
We stick instead to our notation $I_p^w$ for the wild inertia subgroup of $I_p$, and $I_p^t$ for the quotient $I_p/I_p^w$.

By \cite[Proposition 4]{Serre1972}, $I_p^w$ acts trivially under $\ovl\rho_p^\sms$, so that $\ovl\rho_p^\sms$ can be factored through the tame inertia quotient $I_p^t=I_p/I_p^w$. Since $I_p^t$ is pro-cyclic, its action can always be diagonalized, with two distinct cases appearing as in the following.

\begin{prop}\label{prop:serre}\cite[Section 2]{Serre1972}
Exactly one of the following holds: 
\begin{itemize}
\item[$\mathrm{(irr)}$] $\ovl\rho_p$ is irreducible, in which case it is absolutely irreducible and 
\[ \ovl\rho_p\vert_{I_p}\cong\psi^{a+pb}\oplus\psi^{b+pa} \]
for a fundamental character of inertia $\psi\colon I_p\to\F_{p^2}^\times$ and integers $0\le a<b\le p-1$.
\item[$\mathrm{(red)}$] $\ovl\rho_p$ is reducible, in which case its semisimplification $\ovl\rho_p^\sms$ satisfies
\[ \ovl\rho_p^\sms\vert_{I_p}\cong\ovl\vareps_p^a\oplus\ovl\vareps_p^b \]
for the mod $p$ cyclotomic character $\ovl\vareps_p\colon I_p\to\F_p^\times$ and integers $0\le a\le b\le p-2$.
\end{itemize}
\end{prop}
\noindent Both $\psi$ and $\ovl\vareps_p$ are trivial on $I_p^w$, which is compatible with what we already remarked.

\begin{rem}\label{rem:serre}
One can deduce from Proposition \ref{prop:serre} the well-known fact that, if $\ovl\rho_p$ is irreducible, then it is induced by a character of $G_{\Q_{p^2}}$.
\end{rem}

Proposition \ref{prop:serre} can be applied to our setting as follows.

\begin{rem}\label{rem:ab}
In the case $\KK=\QQ$, condition $(ii)$ of Proposition \ref{prop:rhop1} is implied by the stronger condition:
\begin{itemize}
    \item[$(ii')$] if $\ovl\rho_p$ is reducible and $a,b$ are as in Proposition \ref{prop:serre} (red), then $a\not\equiv b\pmod {\frac {p-1} 2}$.
\end{itemize} 
\end{rem}

\subsubsection{Representations with big mod $p$ global image}
In the previous applications, the global Galois representation is always potentially totally ramified, so that a decomposition group is open in the global Galois group and it is not necessary to take a normal closure as in assumption \NC. We give an example where this is not the case and the full generality of \Cref{teo:BforLie} comes into play. We show that if the mod $p$ reduction of a global representation into $\GL_2(\ZZ_p)$ is surjective, then \NC can be satisfied even if the local image is not very large.

\begin{prop}\label{prop:2dim}
Let $\rho:G_\KK\to\GL_2(\ZZ_p)$ be a continuous representation and $v$ a $p$-adic place of $\KK$. Assume that:
\begin{enumerate}[label=(\roman*)]
    \item\label{surj} $\ovl\rho:G_\KK\to\GL_2(\FF_p)$ is surjective;
    \item\label{sca} the $\Q_p$-Lie algebra of $\rho(G_v)$ has dimension at least 2;
\item\label{Iv} $\det\rho(I_v)$ is open in $\Z_p^\times$. 
\end{enumerate}
Then $\rho$ has \NC. If moreover
\begin{enumerate}[label=(\roman*),resume]
\item\label{Gv} $\rho\vert_{G_v}$ is absolutely irreducible,
\end{enumerate}
then $\rho$ has \B.
\end{prop}
\begin{proof}
We choose the filtrations $\{G_n\}$ and $\{\GG_n\}$ in \NC to be those induced by the standard one on $\GL_2(\ZZ_p)$, i.e. we let $\rho_n$ be the composition of $\rho$ with the projection $\GL_2(\ZZ_p)\onto\GL_2(\ZZ/p^n)$, and let $\GG_n=\ker\rho_n,G_n=\ker\rho_n\cap G_v$. 

Let $n\ge 1$. Via the logarithm map we map $G_n$ to a $\ZZ_p$-Lie subalgebra $\fg_n$ of $\fgl_2(\ZZ_p)$. By assumption \ref{sca}, $\fg_n\otimes_{\ZZ_p}\QQ_p$ has dimension at least 2, hence it has nonzero intersection with $\fsl_2(\QQ_p)$. Therefore the intersection $\fg_n^0\coloneqq\fg_n\cap\fsl_2(\ZZ_p)$ is also non-zero. Let $x\ne 0$ be an element of $\fg_1^0$, and let $n_0$ be the smallest positive integer such that $x\notin\fg_{n_0}^0$. Then, for every $n\ge n_0$, $p^{n-n_0}x\in \fg_{n-1}^0\setminus\fg_n^0$, so that $\fg_{n-1}^0/\fg_n^0\ne 0$.

Again the logarithm map, we identify $G_{n-1}/G_n$, $\GG_{n-1}/\GG_{n}$ and $\ovl{G_{n-1}/G_n}$ with additive subgroups (hence $\F_p$-vector subspaces) $V_n$, $\VV_n$ and $\ovl V_n$ of $\fgl_2(\F_p)$. Clearly, $V_n=\fg_{n-1}/\fg_n$, and $\ovl V_n$ is a $G_\KK$-stable subspace of the adjoint representation of $G_\KK$ on $\fgl_2(\F_p)$, hence of the adjoint representation of $\GL_2(\F_p)$ because of assumption \ref{surj}. 


We need to show that $\ovl V_n=\VV_n$ for $n$ large enough. 
Let $V_n^0=V_n\cap \fsl_2(\F_p)$, $\VV_n=\VV_n^0\cap \fsl_2(\F_p)$ and $\overline{V}_n^0= \overline{V}_n\cap \fsl_2(\F_p)$. By our earlier remarks, if $n$ is large enough, $V_n^0=\fg_{n-1}^0/\fg_n^0\ne 0$. Since $\fsl_2(\F_p)$ is a simple representation of $\GL_2(\F_p)$, condition \ref{surj} implies that $\ovl V_n^0=\fsl_2(\F_p)=\VV_n^0$. If $\VV_n^0=\VV_n$, there is nothing left to prove. If not, then $\det$ maps the subgroup $\GG_{n-1}/\GG_n\subset\GL_2(\Z/p^n)$ onto   the $p$-group $(1+p^{n-1}\ZZ)/(1+p^{n}\ZZ)\simeq \FF_p$, hence it is surjective. By assumption \ref{Iv},  if $n$ is sufficiently large, then $\det\rho_n(I_v)$ surjects onto $\det(\GG_{n-1}/\GG_n)$, which shows that $V_n$, and hence $\ovl V_n$ also contains the scalars.

For the second part of the statement, it is enough to check that $\rho_v\coloneqq\rho\vert_{G_v}$ satisfies \PTR and $\rho$ has \CE, and apply \Cref{teo:BforLie}. If $\rho_v$ is strongly absolutely irreducible, then it satisfies \PTR by \Cref{cor:bigimage} and \Cref{prop:inaspettata}. If it is absolutely irreducible, but not strongly, then either it is a twist of a representation with finite image, 
or it is induced by a character $\chi:H\to\cO^\times$ for an open subgroup $H\subset G_v$ of index 2 and the valuation ring $\cO$ of a $p$-adic field (at most quadratic). We can exclude the first case, since by assumption $(iii)$ the $\Q_p$-Lie algebra of $\rho(G_v)$ is at least 2-dimensional, hence the projective image cannot be finite. Therefore, we can assume that we are in the induced case. If $\chi(I_v)$ is finite, then $\det\Ind_{H}^{G_v}(\chi)(I_v)$ is also finite, a contradiction. Therefore, $\chi(I_v)$ is an infinite subgroup of $\cO^\times$. Therefore, the $\Q_p$-Lie algebra of $\rho(I_v)$ is 2-dimensional (it is the Lie algebra of the image of $\cO^\times\into\GL_2(\Z_p)$). Since $\rho_v$ is induced, the $\Q_p$-Lie algebra of $\rho(G_v)$ is also 2-dimensional, so $\rho(I_v)$ is open in $\rho(G_v)$, as desired. 

We show that, under $(iv)$, $\rho$ has \CE, hence \B by \Cref{teo:rhohasBcenter}. If $\rho\vert_{G_v}$ is strongly absolutely irreducible, then $\rho(I_v)$ contains an open subgroup of $\SL_2(\QQ_p)$ by \Cref{cor:bigimage}, and since $\det\rho(I_v)$ is infinite, $\rho(I_v)$ is open in $\GL_2(\ZZ_p)$ by \Cref{prop:inaspettata}. Therefore, $\rho$ has \CE by \Cref{prop:integer}. 

If instead $\rho\vert_{G_v}$ is induced by a character $\chi:G_{\KK'_{v'}}\to\Qp^\times$ for a quadratic extension $\KK'_{v'}$ of $\KK_v$, we show \CE directly. It is enough to show that $\rho(I_v)$ contains infinitely many scalar matrices. Via a local reciprocity map at $v'$, write $\chi\vert_{I_{v'}}\ccirc\mathrm{rec}_{v'}=\eta\prod_{\sigma\in\Sigma'}\langle\sigma\rangle^{a_{\sigma}}:\cO_{\KK'_{v'}}^\times\to\ovl\Z_p^\times$, where $\eta$ a finite order character, $a_\sigma\in\ZZ_p$, $\Sigma'$ is the set of embeddings $\KK'_{v'}\into\Qp$, and similarly to \Cref{prop:integer} we write $\langle\sigma\rangle$ for the homomorphism $\cO_{\KK'_{v'}}^\times\to\ZZ_p^\times$ that coincides with $\sigma$ on the pro-$p$ factor and is trivial on the prime-to-$p$ factor. If $c$ is a generator of $\Gal(\KK'_{v'}/\KK_v)$, then the restriction $\rho\vert_{I_{v'}}$ is the direct sum of $\chi$ and of $\chi^c=(\eta c)\prod_{\sigma\in\Sigma'}\langle\sigma c\rangle^{a_\sigma}$. For $h\in\cO_{\KK'_{v'}}^\times$, in order for $\rho(h)$ to be scalar it is sufficient that $\eta(h)=1$ and $\prod_{\sigma\in\Sigma'}\langle\sigma(\sigma c)^{-1}\rangle(h^{a_\sigma-a_{\sigma c}})=1$. The first condition holds on an open subgroup of $\cO_{\KK'_{v'}}$, the second if and only if $h^{\sum_\sigma(a_\sigma-a_{\sigma c})}\in\cO_{\KK_v}^\times$ (hence trivially for $h\in\cO_{\KK_v}^\times$). It is enough to show that these conditions produce infinitely many scalar matrices in $\rho(I_{v'})$. If $\prod_{\sigma\in\Sigma'}\langle\sigma\rangle^{a_\sigma}$ is nontrivial on $\cO_{\KK_v}^\times$, then it maps $\cO_{\KK_v}^\times$ to an infinite subgroup of $\Qp^\times$, as desired. If not, we must have $a_\sigma=-a_{\sigma c}$ for every $\sigma$. But in this case, $\det\rho\vert_{I_{v'}}=\chi\chi^c=\eta(\eta c)$ is of finite order, contrary to our assumption that $\det\rho(I_v)$ is infinite. 
%
%
%
\end{proof}

\begin{remark}\label{lie2dim}
By \Cref{prop:liealg}, Assumption (ii) of \Cref{prop:2dim} is satisfied if $\rho_v$ is either strongly absolutely irreducible, or induced by a character of infinite order.
\end{remark}

\subsection{Examples in arbitrary dimension}\label{sec:exampleserre}

Thanks to a result of Serre \cite[\S 4, Théorème 3]{Serre1967}, we can prove a result about property \B for representations of $G_\KK$ of arbitrary dimension.

\begin{thm}\label{thm:serreB} Let $v$ be a place of $\KK$ and $G_v$ be a decomposition group at $v$. Let  $\rho_v= G_{v}\to\GL_d(\ZZ_p)$ be a continuous representation.
Assume that:
\begin{enumerate}[label=(\roman*)]
\item $\rho_v$ is strongly absolutely irreducible representation (with the terminology of Section \ref{appendix:bigimage});
\item $\rho_v$ is Hodge--Tate with weights 0 and 1, whose respective multiplicities $n_0$ and $n_1$ are $\ge 1$ and coprime.
\end{enumerate}
Then $\rho_v(I_v)$ contains an open subgroup of $\GL_d(\Z_p)$, and any continuous representation $\rho\colon G_\KK\to\GL_d(\Z_p)$ such that $\rho\vert_{G_v}\cong\rho_v$ has property \B.
\end{thm}

\begin{proof}
We apply \cite[\S 4, Théorème 3]{Serre1967} to $\rho_v$. 
As discussed in Appendix \ref{appendix:bigimage}, $\rho_v$ is strongly absolutely irreducible if and only if $\Q_p^2$ is absolutely irreducible as a module over the Lie algebra of $\rho_v(G_v)$, which is assumption (H$^\ast$ 1) of \cite[\S 4]{Serre1967}. Assumption (H$^\ast$ 2,3) of \cite[\S 4]{Serre1967} coincide with our assumption (ii). Therefore, the image of $\rho_v$ contains an open subgroup of $\GL_d(\Z_p)$ by \cite[\S 4, Théorème 3]{Serre1967}.

The restriction $\rho_v\vert_{I_v}$ admits two distinct Hodge--Tate weights, hence its image cannot contain an open subgroup consisting of of scalar matrices. 
Its determinant $\det\rho_v\vert_{I_v}$ is a Hodge--Tate character of Hodge--Tate weight $n_1\ge 1$, so its image is infinite. Therefore, by Proposition \ref{prop:riehm}, $\rho_v(I_v)$ contains an open subgroup of $\GL_d(\Z_p)$. The conclusion now follows from Theorem \ref{thm:rhohasB}.
\end{proof}

\begin{rem}
By Proposition \ref{prop:twistfull}, if $\rho_v(I_v)$ contains an open subgroup of $\SL_2(\Z_p)$, then the same is true for every twist of $\rho_v$ with a character. In particular, assumption (ii) of Theorem \ref{thm:serreB} can be weakened by replacing 0 and 1 with $k$ and $k+1$ for an arbitrary $k\in\Z$, and even extended to non-Hodge--Tate representations whose Hodge--Tate--Sen weights differ by 1.
\end{rem}

\subsection{Certain pro-$p$ extensions of local type have property \B}\label{sec:chenevier}

Some examples of potentially totally ramified extensions of $\Q$ can be recovered from the work of Chenevier \cite{Chenevier2008}. Let $\KK$ be a number field, and let $\KK(p)$ be the maximal pro-$p$ extension of $\KK$ unramified away from $p$ and $\infty$. Let $\rho:G_\KK\to\GL_d(\Z_p)$ be a continuous representation that factors through $G_\KK\onto\Gal(\KK(p)/\KK)$.

\begin{proposition}\label{prop:chenevier}
Assume that $p$ does not divide the narrow class number of $\KK$ and $\mug_p\not\subset\KK$. Then $\KK(\rho)$ has property \B.
\end{proposition}

\begin{proof}
By \cite[Lemme 1.3]{Chenevier2008}, $\KK(p)/\KK$ is totally ramified, so $\KK(\rho)/\KK$ is also totally ramified. 
Moreover, $\mug_{p^\infty}\not\subset\KK(\rho)$ since $\mug_p\not\subset\KK$, so that $\KK(\rho)$ has \CE by Proposition \ref{prop:CEP}. Therefore, $\KK(\rho)$ has property \B by Theorem \ref{teo:rhohasBcenter}.
\end{proof}

\begin{rem}
The extension $\KK(\rho)/\KK$ is an example of extension \textit{potentially of local type} in the sense of Wingberg \cite{WingbergLoc}, with $S=\{p,\infty\}$.
\end{rem}

\section{Modular examples}

Let $k\ge 2$ and $N\ge 1$ be two integers. 
Let $f$ be a classical cuspidal form of level $\Gamma_1(N)$ and weight $k$, which is an eigenvector for the Hecke operators $T_\ell$, $\ell\nmid N$. Let $\chi:(\Z/N\Z)^\times\to\ovl\QQ^\times$ be the nebentypus of $f$.
For every prime $\ell\nmid N$, we write $a_\ell(f)$ for the $T_\ell$-eigenvalue of $f$. 
Let $\KK_f$ be the Hecke field of $f$, i.e. the number field generated by the eigenvalues $a_\ell(f)$, $\ell\nmid N$. 


Let $\fp$ be a $p$-adic place of $\KK_f$ such that $\KK_{f,\fp}=\QQ_p$. 
Let $\rho_{f,\fp}:G_\Q\to\GL_2(\KK_{f,\fp})=\GL_2(\Q_p)$ be the $\fp$-adic Galois representation attached to $f$ by Eichler--Shimura and Deligne. 
By a standard argument, we can conjugate $\rho_{f,\fp}$ so that it takes values in $\GL_2(\Z_p)$, and we consider it from now on as a representation $G_\Q\to\GL_2(\Z_p)$. We denote with $\ovl\rho_{f,\fp}:G_\Q\to\GL_2(\F_p)$ its mod $p$ reduction. We look for conditions on $f$ guaranteeing that $\rho_{f,\fp}$ has property \B.




Write $\iota:\KK_f\into\KK_{f,\fp}=\QQ_p$ for the natural inclusion. 
Recall that $f$ is called \textit{supersingular} at $\fp$ if $p\nmid N$ and $a_p(f)\in\fp$ (i.e. the $\fp$-adic valuation of $a_p(f)$ is positive). 
In this case, the mod $p$ reduction of $\iota(a_p(f))$ is 0, and by a result of Fontaine, $\ovl \rho_{f,\fp}$ is irreducible if $p\geq k-1$ \cite[Theorem 2.6]{Edixhoven1992}.

\begin{prop}\label{prop:modhasB1}
Assume that $f$ is supersingular at $\fp$ and that $a_p\ne 0$. 
Then $\rho_{f,\fp}(I_p)$ contains an open subgroup of $\SL_2(\Z_p)$, and $\rho_{f,\fp}$ has property \B.
\end{prop}

We briefly recall some facts from $p$-adic Hodge theory. By \cite[Proposition 3.1.1]{Breuil2003}, every 2-dimensional, irreducible crystalline representations of $G_p$ is of the form $V_{k,a_p,\chi}$, with the notation of \textit{loc. cit.}, for some $k\in\Z$, $a_p\in\Qp$ of positive valuation, and crystalline character $\chi$. We have the following.

\begin{lemma}\label{lemma:breuil}
The representation $V_{k,a_p,\chi}$ is an induced representation if and only if $a_p=0$.
\end{lemma}

\begin{proof}
If $a_p=0$, then $V_{k,a_p,\chi}$ is induced by \cite[Proposition 3.1.2]{Breuil2003}. Conversely, if $V_{k,a_p,\chi}$ is induced, then it is isomorphic to a twist of itself by a quadratic character. Then \cite[Proposition 3.1.1]{Breuil2003} implies that $a_p=0$.
\end{proof}

\begin{proof}[Proof of Proposition \ref{prop:modhasB1}]
By \cite[Théorème 6.2.1]{Breuil2003}, the local Galois representation $\rho_{f,\fp}\vert_{G_p}$ is crystalline and irreducible, isomorphic to the representation $V_{k,\iota(a_p(f))\chi(p)^{1/2},\chi^{1/2}}$. Since $a_p\ne 0$, Lemma \ref{lemma:breuil} implies that $\rho_{f,\fp}\vert_{G_p}$ is not induced.  
The Hodge--Tate weights of $\rho_{f,\fp}\vert_{G_p}$ are 0 and $k-1\ne 0$.  
Then the conclusion follows from Proposition \ref{prop:rhop3} and Theorem \ref{thm:propsimplyB}.
\end{proof}

\begin{rem}\label{rem:ap0}\mbox{ }
\begin{enumerate}[label=\textit{\alph*)}]
\item Proposition \ref{prop:modhasB1} also shows that if $\rho_{f,\fp}\vert_{G_p}$ is supersingular and $a_p\not=0$, then $\rho_{f,\fp}\vert_{I_p}$ is absolutely irreducible.
\item Our result is orthogonal to \cite[Theorem 1.4]{AmorosoTerracini2024}, which relies precisely on the fact that $\rho_{f,\fp}\vert_{G_p}$ is induced when $a_p=0$.
\item If $p\nmid N$, $k=2$, and $\KK_f=\Q$, then $f$ corresponds to an elliptic curve, a case already treated in \cite{Habegger2013}. Moreover, the Ramanujan--Petersson bound implies that $a_p=0$ in this case, so that one can also deduce property \B from \cite{AmorosoTerracini2024}. On the other hand, if $\KK_f\ne\Q$, we can have supersingular eigenforms with $a_p\ne 0$ already in weight 2.
%

\item If $f$ is CM, then there exists no prime $p$ satisfying the assumptions of Proposition \ref{prop:modhasB1}: either $p$ is ramified in the field of complex multiplication, in which case it divides the level, or it is inert, in which case $a_p=0$, or it is split, in which case $f$ is ordinary at $p$. This reflects the fact that the global $p$-adic Galois representation attached to $f$ is induced for every $p$, hence no local Galois representation can be of large image.

\item The condition $a_p\ne 0$ is ``often'' satisfied.
If $f$ is a non-CM eigenform, then $a_p=0$ only for a set of primes $p$ of density 0, as a consequence of a proved case of the Sato-Tate conjecture \cite[Corollary 8.4]{BGHT2011}. 
If $k$ and $N$ are fixed, then there exists only a finite number of non-CM eigenforms with $a_p=0$, by \cite[Theorem 1.0.1]{CalegariSardari}.

\item When $f$ is $p$-ordinary, Greenberg (see \cite[Question 1]{Ghate2004}) conjectured that $f$ is CM if and only if the local Galois representation $\rho_{f,\fp}\vert_{G_p}$ is the direct sum of two characters (note that, in order for $f$ to be CM and $p$-ordinary, $p$ must split in the field of complex multiplication, which implies that the global induced representation becomes reducible on a decomposition group). One could ask whether in the $p$-supersingular case it is also possible to detect whether $f$ is CM from $\rho_{f,\fp}\vert_{G_p}$ (i.e. whether assumption (ii) of Proposition \ref{prop:modhasB1} can be replaced with the condition that $f$ is non-CM). However, this is not the case: for instance, in the situation considered in \cite{AmorosoTerracini2024}, $f$ is an eigenform of level prime to $p$ satisfying $a_p=0$, and in this case $\rho_{f,\fp}\vert_{G_p}$ is the crystalline representation $V_{k,0}$, which is induced, but there exist plenty of examples of non-CM eigenforms $f$ with a vanishing Hecke eigenvalue away from the level. 
\end{enumerate}
\end{rem}

\begin{rem}
Unfortunately, one cannot hope for an obvious generalization of Proposition \ref{prop:modhasB1} to representations of $G_\KK$, $\KK$ a number field, attached for instance to elliptic curves over $\KK$ or Hilbert modular forms for totally real $\KK$. Indeed, it is not true that a crystalline representation of $G_{\KK_v}$ will have large image if and only if the trace $a_v$ of the crystalline Frobenius is non-zero: for instance, Habegger shows in \cite[Lemma 3.3]{Habegger2013} that the representation of $G_{\Q_{p^2}}$ attached to a supersingular elliptic curve over $\Q_{p^2}$ with $j$-invariant not congruent to 0 or 1728 modulo $p$ is abelian, whereas $a_v=\pm 2p$, $v$ the unique place of $\Q_{p^2}$ above $p$. 
\end{rem}


\subsection{$\Gamma_0(N)$-regularity and residual reducibility} \label{sec:regular} Keep the above notation. By \cite[Theorem 2.6]{Edixhoven1992}, if $f$ is supersingular at $p$ and $k\leq p+1$, then $\ovl\rho_{f,\fp}\vert_{G_p}$ is absolutely irreducible. On the other hand, it can be of some interest to produce some examples of supersingular modular forms $f$ (with $a_p\not=0$) whose associated $G_p$-representation is reducible modulo $p$: i.e. such that $\rho_{f,\fp}\vert_{G_p}$ has a ``large'' image, although the image of its reduction modulo $p$ is ``small''.
Let $p>3$ be $\Gamma_0(N)$-regular, in the sense of \cite[Definition 1.2]{BuzzardSlopes}. If $f$ has trivial nebentypus modulo $p$, then the determinant of $\ovl\rho_{f}$ is a power of the cyclotomic character, so that,  by \cite[Lemma 1.4]{BuzzardSlopes}, $\ovl\rho_{f,\fp}\vert_{G_p}$ is reducible. 

The above discussion provides a simple recipe to produce examples of supersingular eigenforms such that $a_p\not=0$ (so that $\rho_{f}$ has property \B) and $\ovl\rho_{f}$ is reducible. Start with an arbitrary $k\ge 2$ and $N\ge 1$; we look for a non-CM, cuspidal eigenform $f$ of weight $k$ and level $\Gamma_0(N)$ and a prime $p>3$, $p\nmid N$ in the intersection of the following sets of primes:
\begin{itemize}
    \item the set $S_1$ of primes $p$ such that $f$ is supersingular at $p$ with non-zero $T_p$-eigenvalue;
    \item the set $S_2$ of primes that split completely in $\KK_f$;
    \item the set $S_3$ of $\Gamma_0(N)$-regular primes.
    \end{itemize}
For every such prime, $\rho_{f}|_{G_p}$ will be defined over $\Q_p$ and residually reducible, but its image will contain an open subgroup of $\GL_2(\Z_p)$ because of Proposition \ref{prop:modhasB1}.

The set $S_2$ is infinite, with density $1/[\KK_f:\Q]$ in the set of rational primes. 
Unfortunately, for any given $N$, it is not known whether there exist infinitely many $\Gamma_0(N)$-regular primes (see the introduction of \cite{BuzzardGeeSlopes}). However, a certain amount of examples have been computed. For instance, the first few $\SL_2(\Z)$-irregular primes are 59, 79, 107, 131, 139, 151, 173 \cite[Section 1]{BuzzardSlopes}.

One can search for eigenforms $f$ and primes with the above properties in the LMFDB \cite{LMFDB}. If one looks among the $\SL_2(\Z)$-eigenforms defined over $\Q$, then the splitting condition is empty and every supersingular prime $p<100$ with $p\ne 2,3,59,79$ will do. For instance, one finds the eigenforms 1.12.a.a, 1.16.a.a, 1.18.a.a, 1.20.a.a, 1.22.a.a, 1.26.a.a, that are all supersingular at the primes 5 and 7, with non-zero $T_p$-eigenvalue.

The orbit 1.24.a.a contains two eigenforms defined over the field $\Q(\sqrt{144169})$, in which both 5 and 7 split. The two eigenforms are supersingular at both 5 and 7, with non-zero $T_p$-eigenvalue, so they also satisfy our assumptions at these primes.

The prime $p=5$ is also $\Gamma_0(6)$- and $\Gamma_0(8)$-regular \cite[Section 4]{BuzzardSlopes}, so by looking for $5$-supersingular eigenforms with coefficient field $\Q$, level $\Gamma_0(6)$ or $\Gamma_0(8)$, and weight up to 20, one finds the following examples: 6.12.a.b, 6.12.a.c, 6.16.a.a, 6.16.a.b, 6.18.a.a, 6.18.a.c, 6.20.a.b, 6.20.a.c; 8.8.a.b, 8.10.a.b, 8.12.a.a, 8.14.a.a, 8.16.a.b, 8.16.a.c, 8.18.a.a, 8.20.a.a.

\section{Geometric examples}
\subsection{Elliptic curves over a number field}

We apply \Cref{prop:2dim} to show that the $p^\infty$-torsion of an elliptic curve over a number field generates a field with property \B, for a suitable choice of $p$. Let $E$ be an elliptic curve elliptic defined over a number field $\KK$. When $E$ has CM, then $\KK(E[p^\infty])$ is an abelian extension of a quadratic extension of $\KK$, hence it has \B by \cite[Theorem 1.1]{AmorosoZannier2000}.

\begin{thm}\label{thm:ellK}
Assume that $E$ is non-CM and that there exists a rational prime $p$ such that:
\begin{enumerate}[label=(\roman*)]
\item\label{rhobarsurj} $\ovl\rho_{E,p}:G_\KK\rightarrow \GL_2(\FF_p)$ is surjective; 
\item for a $p$-adic place $v$ of $\KK$, $\rho|_{G_v}$ is absolutely irreducible. 
\end{enumerate} 
Then $\rho_{E,p}$, equivalently $\KK(E[p^\infty])$, has \B.
\end{thm}


\begin{proof}
We apply \Cref{prop:2dim}. Conditions \ref{surj} and \ref{Gv} hold by assumption.
Since $\det\rho$ is the cyclotomic character, \ref{Iv} is clear. 
By Remark \ref{lie2dim}, condition \ref{sca} follows from the absolute irreducibility of $\rho\vert_{G_v}$. 
\end{proof}

\begin{remark}
By a classical result of Serre, given an elliptic curve $E$, condition \ref{rhobarsurj} is satisfied for large enough $p$. 
\end{remark}

\begin{remark}
One might also apply \Cref{prop:2dim} to the case of supersingular eigenforms with $a_p=0$, for which the local Galois representation is induced. However, this case is already dealt with in \cite{AmorosoTerracini2024}.

On the other hand, we hope to upgrade \Cref{thm:ellK} to an adelic result for elliptic curves over a number field in a follow-up work.
\end{remark}


\subsection{Abelian varieties of $\GL_2$-type}
Let $f=\sum_na_nq^n\in S_2(\Gamma_0(N))$ be a normalized Hecke eigenform with Hecke field $\KK_f$. Assume that $[\KK_f:\QQ]\geq 2$. Let $p$ be a prime not dividing $N$ and $\fp$ a $p$-adic place of $\KK_f$ such that $f$ is supersingular at $\fp$, $a_p\not=0$, and $\KK_{f,\fp}=\QQ_p$. Then the representation $\rho_{f,\fp}:G_{\QQ}\to\GL_2(\QQ_p)$ attached to $f$ has \B by \Cref{prop:modhasB1}.

To a form $f$ as above we can attach an abelian variety $A_f$ over $\QQ$, of $\GL_2$-type, with $\KK_f\subset\QQ\otimes_\ZZ\End(A_f)$. The representation of $G_\QQ$ on the Tate module $A_f[p^\infty]$ decomposes as a direct sum of 2-dimensional representations 
\[ \rho_{A_f,\fp}:G_\QQ\to\GL_2(\KK_{f,\fp}) \]
with coefficients in the completions of $\KK_f$ at the $p$-adic places.

\begin{remark}
If $p$ splits completely in $\KK_f$ and $f$ is supersingular at every $p$-adic place of $\KK_f$, then the Hasse bound forces $a_p=0$, so that we are not in the framework of \Cref{prop:modhasB1} anymore. One could try to attack this case via the normal closure assumption as in \Cref{prop:2dim} to prove that the whole $p^\infty$-torsion of $A_f$ generates a field with \B, under suitable assumptions on $A_f$ and $p$. However, this work is essentially carried out in \cite{AmorosoTerracini2024}, where the authors even prove the adelic property \B for $A_f$. 
\end{remark}


\begin{ex}\label{ex:AV}
We extracted from \cite{LMFDB} some numerical examples of forms $f$ and primes $p$ such that $f$ is supersingular at a $p$-adic place $\fp$, hence, with the notation from the previous paragraph, $\rho_{A_f,\fp}$ has \B. We thank Marzio Mula for help with automating this search. We identify newform orbits via their name in \emph{loc. cit.}; in particular, the first entry is the level $N$, the second the weight (equal to 2). We write $P_p(x)$ for the characteristic polynomial of the Hecke operator $T_p$ acting on the orbit.
\begin{itemize}
\item 43.2.a.b, $p=17$: $\KK_f=\QQ(\sqrt{2})$, $p=17$ splits in $K_f$ 
and $P_{17}(x)=x^2-10x+17$. Each form in the orbit is supersingular at one place above 17, and ordinary at the other (if it were supersingular at both places, Ramanujan's bound would force the trace of $T_p$ to be 0). 
\item 53.2.a.b, $p=17$, $P_{17}(x)=x^3+5x^2-5x-17$.
\item 71.2.a.a, $p=3,5$, $P_3(x)=x^3-x^2-4x+3$, $P_5(x)=x^3-5x^2-2x+25$.
\item 71.2.a.b, $p=3,37$, $P_3(x)=x^3+x^2-8x-3$, $P_{37}(x)=x^3-9x^2-26x+37$.
\end{itemize}
In the last three examples, each $p$ is a product $\fp\mathfrak q$ of primes in the Hecke field $\KK$, with $\KK_{\fp}=\QQ_p$ and $\KK_{\mathfrak q}=\QQ_{p^2}$, so we can take $f$ to be the form in the orbit corresponding to the root of $P_p(x)$ that belongs to $\fp$.
\end{ex}

\section{Examples in $p$-adic families}\label{sect:families}

 We exhibit some ``large'' subspaces of deformation spaces for Galois representation and modular forms, whose points correspond to representations of $G_\QQ$ with property \B. The discussion of Section \ref{sec:exampledef} could be extended to the case of $G_\KK$, $\KK$ a number field, at the expense of complicating the notation, while that of Section \ref{sec:overconvergent} could likely be generalized to the case of $p$-adic families of Hilbert modular forms, i.e. to totally real $\KK$.

\subsection{Deformations with property \B}\label{sec:exampledef}
Let $\ovl\rho:G_\Q\to\GL_2(\F_p)$ be a continuous representation. We identify subspaces of a (pseudo-)deformation space for $\ovl\rho$ whose points correspond to representations with property \B.

Set $\ovl\rho_p=\ovl\rho\vert_{G_p}$. 
Let $\F$ be a finite extension of $\F_p$ such that the eigenvalues of all the matrices in the image of $\ovl\rho$ belong to $\F$. From now on, we extend the coefficients of $\ovl\rho$ to $\F$, so that $\ovl\rho$ and $\ovl\rho_p$ are either absolutely irreducible or reducible.

Let $\CNL_\F$ be the category of complete, Noetherian local $W(\F)$-algebras with residue field $\F$.




\subsubsection{The residually irreducible case}\label{sec:irr}
Assume that $\ovl\rho_p:G_p\to\GL_2(\F)$ is irreducible (hence absolutely irreducible). Then, by Remark \ref{rem:serre}, $\ovl\rho_p$ is absolutely irreducible and induced by a character of $G_{\Q_{p^2}}$. In particular, $\ovl\rho$ is also absolutely irreducible, and Mazur's theory \cite{MazurDef1989} produces universal deformations $\rho_p^{\univ}\colon G_\Q\to\GL_2(R_{\ovl\rho_p})$ and $\rho^{\univ}\colon G_\Q\to\GL_2(R_{\ovl\rho})$ over universal deformation rings $R_{\ovl\rho_p}, R_{\ovl\rho}\in\CNL_\F$. Strictly speaking, deformations are equivalence classes of representations, but we fix two representatives $\rho^\univ_p$ and $\rho^\univ$ as above.

Let $R_{\ovl\rho_p}^\ind$ be the quotient of $R_{\ovl\rho_p}$ produced by \cite[Lemma 2.3.2]{CalegariSardari} (and denoted by $R^{\loc,\ind}$ therein); it pro-represents the functor $D^\ind$ of locally induced deformations introduced in \cite[Definition 2.3.1]{CalegariSardari}. Note that an absolutely irreducible $\ovl\rho_p$ is always of the shape given in \cite[Assumption 2.2.1]{CalegariSardari}, and that the condition that $n$ is even in \textit{loc. cit.} is only required in order to assure that $\ovl\rho_p$ is absolutely irreducible, which we are already assuming.

Similarly to \cite[Definition 2.4.2]{CalegariSardari}, set $R_{\ovl\rho}^{\loc-\ind}=R_{\ovl\rho}\widehat{\otimes}_{R_{\ovl\rho_p}}R_{\ovl\rho_p}^\ind$. Let $X^{\loc-\notind}$ be the open complement of $\Spec R_{\ovl\rho}^{\loc-\ind}$ in $\Spec R_{\ovl\rho}^\sq$. By \cite[Theorem 1]{BoeckleJuschka2023}, $R_{\ovl t_p}$ is equidimensional of relative dimension 5 over $W(\FF)$, while a calculation as in \cite[Lemma 2.5.1]{CalegariSardari} (without fixing the character $\psi$) shows that $R_{\ovl\rho}^{\loc-\ind}$ is of relative dimension 3, so that in particular the complement $X^{\loc-\notind}$ is dense.



For every $\Qp$-point $x$ of $\Spec R_{\ovl\rho}$, we denote with $\rho_x:G_\Q\to\GL_2(\Qp)$ the specialization of $\rho^\univ$ at $x$. 

\begin{prop}\mbox{ }\label{prop:induced}
\begin{enumerate}[label=(\roman*)]
\item Let $\rho:G_\Q\to\GL_2(\Zp)$ be a continuous representation lifting $\ovl\rho$. Then $\rho\vert_{G_p}$ is induced if and only if the corresponding map $R_{\ovl\rho}\to\Qp$ factors through $R_{\ovl\rho}\to R_{\ovl\rho}^{\loc-\ind}$.
\item Let $x$ be a $\Qp$-point of $X^{\loc-\notind}$ such that $\rho_x\vert_{G_p}$ can be defined over $\Q_p$ and has distinct Hodge--Tate--Sen weights. Then $\rho_x$ has property \B.
\end{enumerate}
\end{prop}

\begin{proof}
Part \textit{(i)} follows from the universal property of $R_{\ovl\rho}^{\loc-\ind}$.

Let $x$ be as in \textit{(ii)}. Since $\ovl\rho_x\vert_{G_p}=\ovl\rho_p$ is absolutely irreducible, so is $\rho_x\vert_{G_p}$. By definition of $X^{\loc-\notind}$ and part \textit{(i)}, $\rho_x\vert_{G_p}$ is not induced. Then the conclusion follows from Proposition \ref{prop:rhop3} and Theorem \ref{thm:propsimplyB}.
\end{proof}




\subsubsection{The residually reducible case}\label{sec:red}
In order to treat the residually reducible case, we work with pseudorepresentations to avoid problems with. We recall some standard definitions. Let $A$ be a local pro-$p$ ring, with residue field $\F$, and $t:\Pi\to A$ a continuous, 2-dimensional pseudorepresentation (or pseudocharacter), in the sense of \cite[Definition 1.2.1]{BellaicheChenevier}. We say that $t$ is \textit{reducible} if it is the sum of two characters. When $A$ is a field, $t$ is reducible if and only if every representation with trace $t$ is reducible.
We say that $t$ is \textit{residually multiplicity-free} if the reduction $\ovl t$ of $t$ modulo the maximal ideal of $A$ is not the sum of two copies of the same character. Note that if $\ovl t\otimes_{\F}\ovl\F_p$ is the direct sum of two copies of a character $\chi$, then $\chi$ is already defined over $\F$, since $p\ne 2$.

Assume that $\ovl\rho_p$ is absolutely reducible, i.e. becomes reducible over $\F$. Then by Proposition \ref{prop:serre}, it is reducible over $\F_p$ and $\ovl\rho_p^\sms\vert_{I_p}\cong\ovl\vareps_p^a\oplus\ovl\vareps_p^b$ for some integers $a,b$. Let $\ovl t\coloneqq\trace\ovl\rho$ and $\ovl t_p\coloneqq\trace\ovl\rho_p$ be the pseudorepresentations attached to $\ovl\rho$ and $\ovl\rho_p$, respectively. 

Assume that $a\not\equiv b\pmod{\frac{p-1}{2}}$. This implies that $\ovl t_p$ is multiplicity-free, and that if $\cO$ is the valuation ring in a $p$-adic field containing $W(\F)$, and $t$ a lift of $\ovl t_p$ to $\cO$, then $t$ is not the trace of an induced representation by Lemma \ref{lemma:noind}. 

As in \cite[Section 2.3]{BoeckleDef}, we define functors $D_{\ovl t}$ and $D_{\ovl t_p}$ of pseudodeformations of $\ovl t$ and $\ovl t_p$. By \cite[Proposition 2.3.1]{BoeckleDef}, they are representable by universal pairs $R_{\ovl t}\in\CNL_\F, t^\univ:G_\Q\to R_{\ovl t}$ and $R_{\ovl t_p}\in\CNL_\F, t^\univ:G_p\to R_{\ovl t_p}$, respectively.

For every continuous character $\ovl\chi:G_p\to\F^\times$, the functor $D_{\ovl\chi}$ of deformations of $\ovl\chi$ to $\CNL_\F$ is represented by a ring $R_{\ovl\chi}\in\CNL_\F$, isomorphic to $W(\F)[[G_p^{\ab,(p)}]$, where $G^\ab$ denotes the abelianization of $G_p$ and $G_p^{\ab,(p)}$ its maximal pro-$p$ quotient. By local class field theory, $R_{\ovl\chi}$ is of relative dimension 2 over $W(\FF)$.

For every $A\in\CNL_\F$, let $D^{\red}_{\ovl t_p}(A)$ be the set of pseudodeformations $t_p$ of $\ovl t_p$ to $A$ with the property that $t_p=\chi_1+\chi_2$ for two characters $\chi_1,\chi_2:G_p\to A$ reducing to $\ovl\vareps_p^a,\ovl\vareps_p^b$, respectively, modulo the maximal ideal of $A$. This defines a subfunctor $D^{\red}_{\ovl t_p}\subset D_{\ovl t_p}$.

\begin{lemma}
The functor $D^{\red}_{\ovl t_p}$ is representable by a quotient $R^{\red}_{\ovl t_p}$ of $R_{\ovl t_p}$, isomorphic to the product $R_{\vareps_p^a}\widehat{\otimes}_{W(\F)}R_{\vareps_p^b}$.
\end{lemma}

\begin{proof}
The functors $D^{\red}_{\ovl t_p}$ and $D_{\vareps_p^a}\times D_{\vareps_p^b}$ are isomorphic: indeed, for $A\in\CNL_\F$, a reducible pseudorepresentation $t_p:G_p\to A$ decomposes as a sum of two $A$-valued characters lifting $\vareps_p^a,\vareps_p^b$, respectively, and conversely the sum of two such characters is an $A$-valued pseudorepresentation lifting $\ovl t_p$. Therefore, $D^{\red}_{\ovl t_p}$ is pro-represented by $R_{\vareps_p^a}\widehat{\otimes}_{W(\F)}R_{\vareps_p^b}$.
 
Since $D^{\red}_{\ovl t_p}$ is a subfunctor of $D_{\ovl t_p}$, there is a surjection $R^{\red}_{\ovl t_p}\onto R^{\red}_{\ovl t_p}$ between the corresponding representing rings.
\end{proof}

The quotient map $R_{\ovl t_p}\to R^{\red}_{\ovl t_p}$ identifies $\Spec R^{\red}_{\ovl t_p}$ with a closed subscheme of $\Spec R_{\ovl t_p}$. We denote with $X^{\loc-\irr}$ its open complement. If the statements are rephrased in terms of Chenevier's determinant, rather than pseudorepresentations, $X^{\loc-\irr}$ coincides with the ``absolutely irreducible locus'' from \cite[Example 2.20]{Chenevier2014}. By \cite[Theorem 1]{BoeckleJuschka2023}, $R_{\ovl t_p}$ is equidimensional of relative dimension 5 over $W(\FF)$, while $R_{\vareps_p^a}\widehat{\otimes}_{W(\F)}R_{\vareps_p^b}$ is of relative dimension 4, so that $X^{\loc-\irr}$ is dense in $\Spec R_{\ovl t_p}$.

For every $\Qp$-point $x$ of $\Spec R_{\ovl t}$, we denote with $\rho_x:G_\Q\to\GL_2(\Qp)$ any continuous representation with trace the specialization of $t^\univ$ at $x$. 

\begin{prop}\mbox{ }\label{prop:reducible}
\begin{enumerate}[label=(\roman*)]
\item Let $\rho:G_\Q\to\GL_2(\Zp)$ be a continuous representation lifting $\ovl\rho$. Then $\rho\vert_{G_p}$ is reducible if and only if the corresponding map $R_{\ovl\rho}\to\Qp$ factors through $R_{\ovl\rho}\to R_{\ovl\rho}^{\loc-\red}$.
\item Let $x$ be a $\Qp$-point of $X^{\loc-\irr}$ such that $\rho_x\vert_{G_p}$ can be defined over $\Q_p$ and has distinct Hodge--Tate--Sen weights. Then $\rho_x$ has property \B.
\end{enumerate}
\end{prop}

\begin{proof}
Part \textit{(i)} follows from the universal property of $R_{\ovl\rho}^{\loc-\red}$.

Let $x$ be as in \textit{(ii)}. By definition of $X^{\loc-\irr}$ and by \textit{(i)}, $t_x\vert_{G_p}$, hence $\rho_x\vert_{G_p}$, is irreducible, but $\ovl\rho_x\vert_{G_p}=\ovl\rho_p$ is reducible. Given our assumption that $a\not\equiv b\pmod{\frac{p-1} 2}$, the conclusion follows from Proposition \ref{prop:rhop3}, Remark \ref{rem:ab}, and Theorem \ref{thm:propsimplyB}. 
\end{proof}




\begin{rem}
Under our assumption that $a\equiv b\pmod{\frac{p-1}{2}}$, the locus of $x\in\Spec A$ where $t_x$ is induced is empty by Lemma \ref{lemma:noind}. 
One could also combine the results of Sections \ref{sec:irr} and \ref{sec:red} to treat the case when $a\equiv b\pmod{\frac{p-1}{2}}$, the only complication being that both the induced and reducible loci might be non-empty in $\Spec A$.
\end{rem}

\subsection{Overconvergent modular examples}\label{sec:overconvergent}

The assumptions of Theorem \ref{thm:rhohasB} have little to do with $\rho$ being attached to a classical modular eigenform: one can already produce a larger class of examples by $p$-adically deforming classical eigenforms.

Let $k\ge 2$ and $N\ge 1$ be two integers, and assume again that $p$ does not divide $N$. Let $\mathcal E_N$ be the Coleman--Mazur eigencurve of tame level $\Gamma_1(N)$, as constructed in \cite{ColemanMazur,BuzzardEigenvarieties}; it is a rigid space over $\Q_p$, equidimensional of dimension 1. The rigid generic fiber of the universal Hida family of tame level $\Gamma_1(N)$ is a union of connected components of $\cE_N$. Every $\Qp$-point $x$ of $\cE_N$ is equipped with a ``weight'' $\kappa_x$, a continuous character $\Z_p^\times\to\Z_p^\times$. We say that a point $x$ of $\cE_N$ has \textit{weight $k_x$}, for $k_x\in\NN$, if $\kappa_x$ is the unique character $\Z_p^\times\to\Z_p^\times$ mapping $1+p$ to $\chi(1+p)\cdot (1+p)^{k_x}$ for a finite order character $\chi$. 
When $x$ corresponds to a classical eigenform $f_x$ of weight $k_x\in\NN$, it has integer weight $k_x$ and $\chi$ as above is the $p$-part of the nebentypus of $f_x$. 

As in \cite[Section 6.1]{ColemanMazur}, the eigencurve $\cE_N$ admits a map to a Galois deformation space: in particular, every $\Qp$-point $x$ of $\cE_N$ carries a continuous representation $\rho_x:G_\Q\to\GL_2(\Qp)$. We denote by $\ovl\rho_x^\sms$ the semisimplification of the mod $p$ reduction of any $\Zp$-lattice in $\rho_x$; it is independent of the chosen lattice, and it is invariant as $x$ moves along a connected component of $\cE_N$. If $x$ is defined over a $p$-adic field $L$ and $\ovl\rho_x^\sms$ is absolutely irreducible, then $\rho_x$ itself is defined over $L$ by a result of Carayol. 

\begin{prop}\label{prop:overconv}
Let $x$ be a $\Q_p$-point of $\cE_N$, and assume that:
\begin{enumerate}[label=(\roman*)]
\item $x$ belongs to a non-ordinary connected component of $\cE_N$;
\item $x$ is not of weight 1;
\item $\rho_x\vert_{G_p}$ is absolutely irreducible and not induced.
\end{enumerate}
Then $\rho_x(I_p)$ contains an open subgroup of $\GL_2(\Q_p)$, and $\rho_x$ has property \B.
\end{prop}

\begin{proof}
We apply Proposition \ref{prop:rhop3} and Theorem \ref{thm:propsimplyB}: it is enough to check the assumption on the Hodge--Tate--Sen weights of $\rho_x\vert_{G_p}$. One of the weights is constantly 0 along the eigencurve, and the other one is the weight of $\kappa_x\id^{-1}$, where $\id:\Z_p^\times\to\Z_p^\times$ is the identity character, hence it is non-zero since $\kappa_x$ is not of weight 1.
\end{proof}

\begin{remark}
It is not clear how many points of a non-ordinary irreducible component of the eigencurve are defined over $\Q_p$. Such components are never finite over the weight space by \cite{HattoriNewton}, and one cannot easily give equations for them.  As pointed out in the Introduction, in future work we plan to remove the condition on the coefficients of our representations to be $\Z_p$, which would render this question irrelevant.
\end{remark}

\subsubsection{Pseudorepresentations on the non-ordinary eigencurve}\label{sec:eigencurve}
We rely on the results of Section \ref{sec:exampledef} to show that condition \textit{(iii)} of Proposition \ref{prop:overconv} is verified almost everywhere on the non-ordinary locus of the eigencurve. 



We equip a rigid analytic space $X$ with the analytic Zariski topology, and we say that a subset of $X$ is discrete if its complement is a dense open.	Note that a discrete set intersects every affinoid in $X$ on a finite number of points.

Let $X$ be a non-ordinary irreducible component of $\cE_N$, $\cO(X)$ the ring of rigid analytic functions on $X$, $\cO^\circ(X)$ the subring of functions of norm bounded by 1, and $t:G_\Q\to\cO^\circ(X)$ be its associated pseudorepresentation. Recall that $\cO^\circ(X)$ is a pro-$p$ local ring by \cite[Lemma 7.2.11]{BellaicheChenevier}, and that points of $X(L)$, $L$ a $p$-adic field, correspond to height 1 prime ideals of $\cO(X\times_{\Q_p}L)$ by \cite[Proposition 1.1]{LoefflerDens2011}.

For $x\in X(\Qp)$, we write $t_x:G_\Q\to\Qp$ for the specialization of $t$ at $x$ and $\rho_x:G_\Q\to\GL_2(\Qp)$ for any representation of trace $t_x$.

Let $\ovl t$ be the reduction of $t$ modulo the maximal ideal of $\cO^\circ(X)$, and let $\ovl t_p=\ovl t\vert_{G_p}$. Let $\ovl\rho:G_\Q\to\GL_2(\F)$ be a continuous representation with trace $\ovl t$, where $\F$ is chosen so as to contain all of the eigenvalues of the matrices in the image of $\ovl\rho$, and set $\ovl\rho_p=\ovl\rho\vert_{G_p}$. Then, by Proposition \ref{prop:serre}, $\ovl\rho_p$ is either absolutely irreducible, or reducible with $\ovl\rho_p^\sms\vert_{I_p}\cong\vareps_p^a\oplus\vareps_p^b$ for $a,b\in\Z$.

\begin{prop}\label{prop:overconvergent}
If $\ovl\rho_p$ is reducible, assume that the corresponding integers $a,b$ satisfy $a\not\equiv b\pmod{\frac{p-1} 2}$. 
Then the set $S$ of points $x\in X(\Qp)$ such that $\rho_x\vert_{G_p}$ is either reducible or induced is discrete. In particular, $\rho_x$ has \B for every $\Q_p$-point $x$ in the complement of a discrete set.
\end{prop}

\noindent The proposition is obviously false if $X$ is ordinary, since then $\rho_x\vert_{G_p}$ is reducible for every $x\in X(\Qp)$. 

\begin{proof}
We first show that $S$ is Zariski-closed in $X$. 
Assume first that $\ovl\rho_p$ is (absolutely) irreducible, so that $t$ is the trace of a representation $\rho:G_\Q\to\GL_2(\cO^\circ(X))$ by \cite[Théorème 1]{Nyssen1996}. Let $R_{\ovl\rho}\to\cO^\circ(X)$ be the map provided by the universal property of $R_{\ovl\rho}$. The locus of $x\in X(\Qp)$ where $\rho_x\vert_{G_p}$ is reducible is empty. On the other hand, by Proposition \ref{prop:induced}, $\rho_x\vert_{G_p}$ is induced if and only if $x$ is in the vanishing locus of $\ker(\cO^\circ(X)\to\cO^\circ(X)\widehat{\otimes}_{R_{\ovl\rho}}R_{\ovl\rho}^{\loc,\ind})$.

Assume instead that $\ovl t_p$ is reducible, and let $R_{\ovl t}\to\cO^\circ(X)$ be the map provided by the universal property of $R_{\ovl t}$. Then, because of our assumption on $a,b$ and Remark \ref{rem:ab}, the locus of $x\in X(\Qp)$ where $\rho_x\vert_{G_p}$ is induced is empty, while by Proposition \ref{prop:reducible}, $\rho_x\vert_{G_p}$ is reducible if and only if $x$ is in the vanishing locus of $\ker(\cO^\circ(X)\to\cO^\circ(X)\widehat{\otimes}_{R_{\ovl t}}R_{\ovl t}^{\loc,\red})$.


We show that $S$ discrete. Since $X$ is 1-dimensional and $S$ is Zariski-closed, this amounts to showing that $S\neq X(\Qp)$. The set of classical points is dense in $X$, while the following sets are discrete:
\begin{itemize}
\item the set $S_1$ of classical points corresponding to $p$-new eigenforms;
\item the set $S_2$ of classical points corresponding to $p$-stabilizations of eigenforms with $a_p=0$;
\item the set $S_3$ of classical points corresponding to $p$-stabilizations of ordinary eigenforms.
\end{itemize}
This follows from slope considerations: if $x$ is a classical point of weight $k$ and slope $h$, then $h=\frac{k-2}{2}$ when $x\in S_1$, $h=\frac{k-1}{2}$ when $x\in S_2$, and $h=k-1$ when $x\in S_3$, and such conditions identify a discrete locus since the slope is locally constant on $X$. In particular, there exists a classical point $x\in X(\Qp)$ that is not in $S_1\cup S_2\cup S_3$, hence satisfies the conditions of Proposition \ref{prop:modhasB1}. The corresponding eigenform is the $p$-stabilization of a supersingular eigenform of level prime to $p$, with $a_p\ne 0$. In particular, $\rho_x\vert_{G_p}$ is irreducible and non-induced by Lemma \ref{lemma:breuil}, so that $x\notin S$.

	
The conclusion now follows from Proposition \ref{prop:overconv}, observing that the set of weight 1 points of $X$ is also discrete, since $X$ is locally finite over the weight space.
\end{proof}

\appendix
\renewcommand{\thetheorem}{\thesubsection.\arabic{theorem}}
\section{$p$-adic Lie groups and Sen's theorem}\label{sect:lie}
\subsection{Basics on $p$-adic Lie groups}

We start this section by recalling some standard results on $p$-adic Lie groups, for which we refer to \cite{AnalyticpropBook} and \cite{Probst2008}.

For a topological group $G$, let $G^p$ denote the subgroup of $G$ generated by $p$-th powers. Given two subgroups $H_1,H_2$ of $G$, we write $[H_1,H_2]$ for their group of commutators. 

As in \cite[Definition 1.15]{AnalyticpropBook}, we define the lower $p$-series of a topological group $G$ as
 \[G_0=G\supseteq G_1=\overline{G^p[G_0,G]}\supseteq \dots\supseteq G_n\supseteq G_{n+1}=\overline{G^p[G_n,G]}\supseteq\ldots \]

\begin{definition}[{\cite[Definition 2.1]{AnalyticpropBook}}] A pro-$p$-group $G$ (not necessarily finitely generated) is \textup{powerful} if
$p$ is odd and $G/\overline{G^p}$ is abelian, or if $p = 2$ and $G/\overline{G^4}$ is abelian.
\end{definition}

For a finitely generated powerful group $G$, \cite[Theorem 2.7]{AnalyticpropBook} gives that
$$G^{p^n}=\{g^{p^n}\ |\ g\in G\},$$
and that the lower $p$-series can be simplified to
\[ G_0=G\supseteq G_1=G^p\supseteq \dots\supseteq G_n=G^{p^n}\supseteq \ldots \]

 \begin{definition}[{\cite[Definition 4.1]{AnalyticpropBook}}]
A finitely generated pro-$p$-group $G$ is \textup{uniform} if $G/{\overline{G}^p}$ is abelian, and all successive quotients $G_n/G_{n+1}$ have the same size.
\end{definition}


We refer to \cite[Definition 8.14]{AnalyticpropBook} for the definition of an \textit{analytic pro-$p$ group}. Following \cite{Probst2008,Sen1972}, we prefer the terminology \textit{$p$-adic Lie group} in this text.

\begin{theorem}[{\cite[Theorem 8.18]{AnalyticpropBook}}]\label{thm:Liepowerful} A topological group $G$ has the structure of a $p$-adic Lie group if and
only if $G$ has an open subgroup that is a powerful finitely generated pro-$p$ group.
\end{theorem}

\begin{theorem}[{\cite[Theorem 8.32]
{AnalyticpropBook}}]\label{teo:Lieuniform} \label{teo:pLie-characterization} Let $G$ be a topological group. Then $G$ is a $p$-adic Lie group
if and only if $G$ contains an open subgroup which is a uniform pro-$p$ group.
\end{theorem}

\begin{theorem}\label{thm:openGL}
A topological group $G$ is a $p$-adic Lie group if and only if it contains an open subgroup isomorphic (as a $p$-adic Lie group) to a closed subgroup of $\GL_d(\Z_p)$ for some $d\ge 0$.
\end{theorem}

\begin{proof}
We sketch a proof as we could not find an exact reference. If $G$ is a topological group containing an open subgroup $H$ as in the statement, then $H$ contains an open uniform pro-$p$ subgroup, hence $G$ is a $p$-adic Lie group by Theorem \ref{teo:Lieuniform}. Vice versa, if $G$ is a $p$-adic Lie group, then it contains an open uniform pro-$p$ subgroup $H$, and by \cite[Theorem 7.19]{AnalyticpropBook} $H$ admits a faithful representation $\iota_d:H\to\GL_d(\Z_p)$ for some $d\ge 0$. Note that since both $H$ and $\GL_d(\Z_p)$ are equipped with the profinite topology, $\iota_d$ is automatically continuous and closed, i.e. it is a homeomorphism between $H$ and a closed subgroup of $\GL_d(\Z_p)$.
\end{proof}

We stick to the case of odd $p$ in the rest of the section, since this is the running assumption in the rest of the paper.

Let $K$ be a $p$-adic field with valuation ring $\cO$.
The following is a standard fact.

\begin{lemma}\label{lemma:lie}
Every compact subgroup $G$ of $\GL_d(K)$ (i.e., every closed subgroup of $\GL_d(\cO))$ is a $p$-adic Lie group. 
\end{lemma}
\begin{proof}
If $n=[K\colon\Q_p]$, choosing an isomorphism of $\Q_p$-vector spaces $K^d\cong\Q_p^{nd}$ and letting elements of $\GL_d(K)$ act $\Q_p$-linearly on $K^d$ gives an injective, continuous group homomorphism $\iota\colon\GL_d(K)\into\GL_{nd}(\Q_p)$. 
Since $G$ is compact and $\GL_d(\Q_p)$ is separated, $\iota\vert_G\colon G\to\iota(G)$ is an isomorphism of topological groups, and $\iota(G)$ is closed in the $p$-adic Lie group $\GL_{nd}(\Q_p)$, hence a $p$-adic Lie group itself by \cite[Theorem 9.6]{AnalyticpropBook}.
\end{proof}

The following lemma provides a converse of 
Lemma \ref{lemma:lie}. It is a consequence of 
Ado's theorem on the existence of a faithful linear representation for any finite dimensional Lie algebra over a field of characteristic 0, see e.g. \cite[\S 7, Théorème 3]{BourbakiLieI}. A proof can be found in  \cite[Proposition 4]{Lubotzky1988}.

\begin{lemma}\label{lem:berger} Every $p$-adic compact Lie group $G$ can be embedded in $\GL_d(\ZZ_p)$, for $d$ large enough.
\end{lemma}

\begin{definition}\label{def:congsub}
The \textup{(principal) congruence subgroups} of $\GL_d(\ZZ_p)$, respectively $\SL_d(\ZZ_p)$, are the subgroups of the form 
\[ \wtl\Gamma_d(p^n)=\ker(\GL_d(\ZZ_p)\to\GL_d(\ZZ/p^n\ZZ)), \]
respectively
\[ \Gamma_d(p^n)=\ker(\SL_d(\ZZ_p)\to\SL_d(\ZZ/p^n\ZZ))=\wtl\Gamma_d(p^n)\cap \SL_d(\ZZ_p), \]
for $n\in\NN$, where in both cases the map is reduction modulo $p^n$. 
\end{definition}

Since principal congruence subgroups are sufficient to our purposes, we always omit the ``principal’’. In particular, congruence subgroups are normal subgroups.

Following \cite{Sen1972} and \cite[Section 6]{Probst2008}, we say that a filtration 
\begin{equation*}G_0\coloneqq G\supseteq G_1\supseteq\ldots\supseteq G_k\supseteq\ldots\end{equation*}
of a $p$-adic Lie group $G$ is a \textit{Lie filtration} if, for large enough $n$, $G_n$ is an open uniform subgroup of $G$ and $G_n\supseteq G_{n+1}\supseteq\ldots$ is the lower $p$-filtration of $G_n$. 
Then for $n\gg 0$ the successive quotients satisfy  $G_n/G_{n+1}\simeq (\Z/p\Z)^A$, where $A$ is independent on $n$; $A$ is called the \textit{dimension} of $G$.

We give two standard examples of $p$-adic Lie groups. We also refer the reader to the similar discussion in \cite[Section 5.1]{AnalyticpropBook}. Recall that $p$ is an odd prime. 

\begin{ex}\label{ex:congruencetilde}\mbox{ }
\begin{itemize}
\item For every $n\ge 1$, $\wtl\Gamma_d(p^n)^p=\wtl\Gamma_d(p^{n+1})$. Indeed, if $\mathcal{B}=1+\mathcal{A}$ with $\mathcal{A}\in p^n\Mat_d(\Z_p)$, then $\mathcal{B}^p=1+p\mathcal{A}(1+\frac{p-1}{2}\mathcal{A}+\ldots)$ is clearly in $\wtl\Gamma_d(p^{n+1})$, so that $\wtl\Gamma_d(p^n)^p\subset\wtl\Gamma_d(p^{n+1})$. Conversely, if $\mathcal{B}=1+\mathcal{A}$ with $\mathcal{A}\in p^{n+1}\Mat_d(\Z_p)$, then $\mathcal{B}^{1/p}$ admits a Taylor series expansion in $\mathcal{A}/p$, with constant coefficient equal to 1, hence belongs to $\wtl\Gamma_d(p^n)$. Note that $\wtl\Gamma_d(p^n)^p=\wtl\Gamma_d(p^{n+1})$ is obviously closed in $\wtl\Gamma_d(p^n)$.
\item For every $n\ge 1$, $\wtl\Gamma_d(p^n)$ is a uniform pro-$p$ group, hence a compact $p$-adic Lie group. Indeed, it is clearly pro-$p$, and $\wtl\Gamma_d(p^{n})/\wtl\Gamma_d(p^n)^p=\wtl\Gamma_d(p^n)/\wtl\Gamma_d(p^{n+1})$ is simply the additive group $\Mat_d(\Z/p\Z)$ (just map $1+p^n\mathcal{A}\in\wtl\Gamma_d(p^{n})$ to $\mathcal{A}$ mod $p$), so that $\wtl\Gamma_d(p^n)$ is a powerful group. Since $\wtl\Gamma_d(p^{n})$ is a finite index subgroup of the finitely presented group $\GL_d(\Z_p)$, it is itself finitely presented. In particular, it is finitely generated, so that its lower $p$-series is 
\[ \wtl\Gamma_d(p^n)\supseteq\wtl\Gamma_d(p^{n+1})\supseteq\ldots\supseteq\wtl\Gamma_d(p^{n+k})\supseteq\ldots \]
by the previous point. Since $\wtl\Gamma_d(p^{n+k})/\wtl\Gamma_d(p^{n+k+1})$ is the additive group $\Mat_d(\Z/p\Z)$ for every $k$, $\wtl\Gamma_d(p^{n})$ is uniform.
\end{itemize}
\end{ex}

\begin{ex}\label{ex:congruence}\mbox{ }
\begin{itemize}
\item For every $n\ge 1$, $\Gamma_d(p^n)^p=\Gamma_d(p^{n+1})$. This follows by taking intersections with $\SL_d(\Z_p)$ in the first point of Example \ref{ex:congruencetilde}. As before, $\Gamma_d(p^n)^p=\Gamma_d(p^{n+1})$ is obviously closed in $\Gamma_d(p^n)$.
\item For every $n\ge 1$, $\Gamma_d(p^n)$ is a uniform pro-$p$ group, hence a compact $p$-adic Lie group. Indeed, it is clearly pro-$p$, and $\Gamma_d(p^{n})/\Gamma_d(p^n)^p=\Gamma_d(p^n)/\Gamma_d(p^{n+1})$ is simply the additive group $\Mat^0_d(\Z/p\Z)$ of trace 0 matrices with coefficients in $\Z/p\Z$ (again, one simply maps $1+p^n\mathcal{A}\in\Gamma_d(p^{n})$ to $\mathcal{A}$ mod $p$), so that $\Gamma_d(p^n)$ is a powerful group. Since $\Gamma_d(p^{n})$ is a finite index subgroup of the finitely presented group $\SL_d(\Z_p)$, it is itself finitely presented. In particular, it is finitely generated, so that its lower $p$-series is 
\[ \Gamma_d(p^n)\supseteq\Gamma_d(p^{n+1})\supseteq\ldots\supseteq\Gamma_d(p^{n+k})\supseteq\ldots \]
by the previous point. For every $k$, $\Gamma_d(p^{n+k})/\Gamma_d(p^{n+k+1})$ is the additive group $\Mat_d^0(\Z/p\Z)$ of trace 0 matrices with coefficients in $\Z/p\Z$, hence $\Gamma_d(p^{n})$ is uniform.
\end{itemize}
\end{ex}

The following lemma is an obvious consequence of Example \ref{ex:congruence}.

\begin{lemma}\label{lemma:filZp}
Let $G$ be a closed subgroup of $\GL_d(\ZZ_p)$. Assume that there exists a positive integer $n$ such that $G$ contains $\wtl\Gamma_d(p^n)$ as a subgroup of finite index. Then
\begin{equation*} G\supseteq\wtl\Gamma_d(p^{n})\supseteq\ldots\supseteq\wtl\Gamma_d(p^{n+k})\supset\ldots \end{equation*}
is a Lie filtration on $G$.
\end{lemma}

\subsection{Ramification groups and Sen's Theorem}\label{appendix:ramification}

We recall some standard results on upper and lower ramification groups, following \cite[Section IV.3]{Serre1968}. Given a finite Galois extension  $L/F$ in $\Qbar_p$, and for an integer $i\geq -1$, let $G[i]$ denote the $i$-th ramification subgroup of $\Gal(L/F)$ \cite[Ch. IV, \S 1]{Serre1968}: then $G[-1]=G=\Gal(L/F)$; and for $i\geq  0$  
\[G[i]=\{\sigma\in G\ |\ \sigma(x)\equiv x\mod \mathfrak{P}^{i+1}, \hbox{ for every } x\in\calo_L\},\] 
where $\mathfrak{P}$ is the maximal ideal in $\calo_L$. For $u\in [-1,\infty)$, we set $G[u]= G[\lceil u\rceil ] $, where $\lceil u \rceil$ is the smallest integer $\geq u$. As in \cite[Ch. IV, \S3]{Serre1968}, we define the function $\varphi_{L/F}\colon[0,+\infty)\to[0,+\infty)$ and its inverse $\psi_{L/F}$. They  allow to define the \textit{upper numbering} of the ramification groups: \[G(u)=G[\psi_{L/F}(u)], \quad\hbox{ that is } G[u]=G(\varphi_{L/F}(u)).\]
\\
If $m\le u\le m+1$ for   where $m$ is a positive integer, then
\begin{equation}\label{eq:phi} \varphi_{L/F}(u)=\frac{1}{g_0}(g_1+\ldots+g_m+(u-m)g_{m+1})\end{equation}
where $g_i=|G[i]|$ for $i\in\NN$.\\ 
Let $L/K/F$ be a tower of Galois extensions. By definition, the lower notation is preserved by passing to subgroups, that is for every $u\in [-1,\infty)$
\begin{equation*}\label{eq:subgroups}  \Gal(L/K)[u]=\Gal(L/F)[u]\cap \Gal(L/K).\end{equation*}
On the other hand, from \cite[Prop. 14, Section IV.3]{Serre1968} we see that the upper notation is preserved by passing to quotients, that is 
\begin{equation}\label{eq:quotients} \Gal(K/F)(u)= (\Gal(L/F)(u)\Gal(L/K))/\Gal(L/K). \end{equation}
The latter is a consequence of the composition formulas 
\begin{equation}\label{comp} \varphi_{L/F}=\varphi_{K/F}\ccirc\varphi_{L/K}\text{ and }\psi_{L/F}=\psi_{L/K}\ccirc\psi_{K/F} \end{equation}
from \cite[Prop. 15, Section IV.3]{Serre1968}. \\
Thanks to  \eqref{eq:quotients}, the upper notation makes sense for infinite Galois extensions, i.e. one can define for any Galois extension $L/F$ in $\Qbar_p$ and $u\in[-1,\infty)$, a ramification group
\[\Gal(L/F)(u)=\varprojlim_{L'}\Gal(L'/F)(u)\]
where $L'$ varies in the set of finite Galois extensions of $F$ contained in $L$.

Assume now that $L/F$ is a Galois extension in $\Qbar_p$, such that $G=\Gal(L/F)$ is a $p$-adic Lie group. Then $G$ is equipped in a natural way with: 
\begin{itemize}
\item a Lie filtration of $G$, obtained by considering its $p$-adic analytic structure;
\item the filtration given by the upper ramification groups.
\end{itemize}
When $L/K$ is totally ramified, Serre \cite{Serre1967} asked about the relation between the two filtrations; the answer was provided by the following:

\begin{theorem}[Sen's Theorem \cite{Sen1972}]\label{teo:sen} Let $\QQ_p\subseteq F \subseteq L\subseteq \Qbar_p$ be a tower of fields such that $L/F$ is totally ramified and  $F/\QQ_p$ has a  finite ramification index $e = e(F/\QQ_p)$.
Assume that the Galois group $G = Gal(L/F)$ is a $p$-adic Lie group, with $dim(G) > 0$.
Let 
$$\ldots G_n\subseteq G_{n-1}\subseteq \ldots\subseteq G=\Gal(L/F)$$ be a filtration on $G$ that is uniformly equivalent in scaling 1 to some Lie
filtration. The filtration $(G_n)_n$ is then uniformly equivalent in scaling $e$ to the ramification filtration
$(G(n))_n$. 
In other words, there exists a constant $c > 0$ such that
$$G(ne + c) \subseteq  G_n\subseteq  G(ne - c)$$
for all $n$, with $G(r) = G$ for $r < 0$.
\end{theorem}

\begin{remark}
In the original version by Sen it is assumed that the base extension $F/\QQ_p$ is finite. This assumption is shown to be unnecessary in \cite{HajirMaire2002,Probst2008}.
\end{remark}

\section{Big image results}\label{appendix:bigimage}

Let $d\ge 1$ be an integer. 
We denote by $\fgl_d$ and $\fsl_d$ the $\Q_p$-Lie algebras of $\GL_d(\Q_p)$ and $\SL_d(\Q_p)$, respectively. 

Let $G$ be a compact (hence a $p$-adic Lie) subgroup of $\GL_d(\Q_p)$, and let $\fg\subset\fgl_d$ be its $\Q_p$-Lie algebra. The natural action of $\GL_d(\Q_p)$ and $\fgl_d$ on $\Q_p^d$ makes $\Q_p^d$ into both a $G$- and a $\fg$-module (the two actions being compatible). 
We extend these actions $\Qp$-linearly to $\Qp^d$.
We sometimes see $G$ as a subgroup of $\GL_d(\Qp)$ via the natural embedding $\GL_d(\Q_p)\hookrightarrow\GL_d(\Qp)$.
For every $\Q_p$-Lie algebra $\fh$, we write $\fh_{\Qp}=\fh\otimes_{\Q_p}\Qp$.

We say that $\Q_p^d$ is an \textit{absolutely irreducible} $G$-module if its $\Qp$-linear extension 
$\Qp^d$ is an irreducible $G$-module. We say that $\Q_p^d$ is a \textit{strongly absolutely irreducible} $G$-module if it is an absolutely irreducible $H$-module for every open subgroup $H\subseteq G$.

The exponential map identifies an open neighborhood of 0 in $\fg$ (equipped with the operation given by the Baker--Campbell--Hausdorff formula) with an open subgroup of $G$. In particular, $\Q_p^d$ is a strongly absolutely irreducible $G$-module if and only if it is an absolutely irreducible $\fg$-module.

Given a subgroup $H\subseteq\GL_d(\Qp)$ or a subalgebra $\fh\subseteq\fgl_{d,\Qp}$, we write $N_{\Qp}(H)$ and $N_{\Qp}(\fh)$ for their respective normalizers in $\GL_d(\Qp)$.

\begin{rem}
If $G$ contains an open subgroup of $\SL_d(\ZZ_p)$, then $G$ is strongly absolutely irreducible, since every open subgroup of $\SL_d(\ZZ_p)$ acts irreducibly on $\Qp^d$. The converse holds if $d=2$ by Proposition \ref{prop:liealg}, but not for $d>2$: for instance $\mathrm{Sym}^{d-1}\SL_2(\ZZ_p)$ acts strongly absolutely irreducibly on $\Q_p^d$, but does not contain any open subgroup of $\SL_d(\ZZ_p)$ if $d>2$.
\end{rem}

\subsection{The 2-dimensional case}
In this subsection, we assume that $d=2$. Keeping the above notation, we prove the following.

\begin{prop}\label{prop:liealg}
Exactly one of the following holds:
\begin{enumerate}[label=(\roman*)]
\item The $G$-module $\Q_p^2$ is strongly absolutely irreducible, in which case $\fsl_2\subset\fg$ and $G$ contains an open subgroup of $\SL_2(\Q_p)$.
\item The $G$-module $\Q_p^2$ is absolutely irreducible, but not strongly absolutely irreducible, in which case $\fg_{\Qp}$ is contained in the Lie algebra of a maximal torus $T$ in $\GL_2(\Qp)$, and:
\begin{enumerate}[label=(\alph*)]
\item if $\fg$ is 2-dimensional, then $G\subset N_{\Qp}(T)$; 
\item if $\fg$ is 1-dimensional, then either $G\subset N_{\Qp}(T)$ and $\det G$ is finite, or $G\cap\SL_2(\Q_p)$ is finite (equivalently, the image of $G$ under $\GL_2(\Q_p)\onto\PGL_2(\Q_p)$ is finite) and $\det G$ is infinite;
\item if $\fg_{\Qp}=0$, then $\fg=0$ and $G$ is a finite group that acts irreducibly on $\Qp^2$.
\end{enumerate}
\item the $G$-module $\Q_p^2$ is not absolutely irreducible, in which case $\fg_{\Qp}$ is a contained in a Borel subalgebra of $\fgl_{2,\Qp}$ and $G$ is contained in a Borel subgroup of $\GL_2(\Qp)$.
\end{enumerate}
\end{prop}

\begin{proof}
Assume first that $\Q_p^2$ is a strongly absolutely irreducible $G$-module, hence an absolutely irreducible $\fg$-module. 
By Lie's theorem, if the $\Qp$-Lie algebra $\fg_{\Qp}$ is solvable, then its irreducible representations are 1-dimensional. In particular, $\fg_{\Qp}$ is non-solvable since it acts irreducibly on $\Qp^2$. Now $\fg_{\Qp}$ is a non-solvable $\Qp$-Lie subalgebra of $\fgl_{2,\Qp}$ dimension at most 4, hence it is either $\fgl_{2,\Qp}$ or $\fsl_{2,\Qp}$, so that $\fg$ is either $\fgl_2$ or $\fsl_2$ and we are in case ($i$).

If the $G$-module $\Q_p^2$ is absolutely irreducible, but not strongly, then $\Qp^2$ is a reducible $\fg_{\Qp}$-module. Since $\Qp^2$ is an irreducible (hence semisimple) $G$-module, then $\Qp^2$ is also semisimple as a $\fg$-module, by the same argument as in (a)$\iff$(b) of \cite[Proposition 1]{Serre1967}. Therefore, $\Qp^2$ is the direct sum of two irreducible $\fg_{\Qp}$-modules, which implies that $\fg_{\Qp}$ embeds into the Lie automorphisms of $\Qp^2$ preserving the two factors, i.e. into the Lie algebra $\ft$ of a $\Qp$-maximal torus $T$. We are then in case ($ii$) of the proposition.

The action of $G$ by conjugation on $\fgl_{2,\Qp}$ preserves $\fg_{\Qp}$. We distinguish 3 cases:
\begin{itemize}
\item If $\fg$ is 2-dimensional, then $\fg_{\Qp}=\ft$. Then conjugation by $G$ preserves $\ft$, which implies that $G\subset N_{\Qp}(\ft)=N_{\Qp}(T)$.
\item If $\fg$ is 1-dimensional, then $\fg_{\Qp}$ is a 1-dimensional subalgebra of $\ft$. If $\fg_{\Qp}$ is contained in the subalgebra of scalars in $\ft$, then $G$ contains an open subgroup consisting of scalar matrices, which implies that $G\cap\SL_2(\Q_p)$ is finite (equivalently, the image of $G$ under $\GL_2(\Q_p)\onto\PGL_2(\Q_p)$ is finite), while $\det G$ is infinite otherwise $G$ would be finite. 
If instead $\fg_{\Qp}$ is a non-scalar 1-dimensional subalgebra of $\ft$, then a direct calculation shows that $N_{\Qp}(\fg_{\Qp})=N_{\Qp}(\ft)=N_{\Qp}(T)$, so that $G\subset N_{\Qp}(T)$.
\item If $\fg=0$, then $G$ is finite.
\end{itemize}

If the $G$-module $\Qp^2$ is reducible, then $\fg_{\Qp}$ embeds into the Lie automorphisms of $\Qp^2$ preserving a flag, that form a Borel subgroup of $\GL_2(\Qp)$.
\end{proof}


Let $\Pi$ be a profinite group and $\rho\colon \Pi\to\GL_2(\Q_p)$ a continuous representation. We say that $\rho$ is strongly absolutely irreducible if $\rho(\Pi)$ is, i.e. if $\rho\vert_H$ is absolutely irreducible for every open subgroup $H\subset G$. 

\begin{cor}\label{cor:bigimage}
At least one of the following holds:
\begin{enumerate}[label=(\roman*)]
\item $\rho$ is strongly absolutely irreducible, in which case $\rho(\Pi)$ contains an open subgroup of $\SL_2(\Q_p)$;
\item $\rho\otimes_{\Q_p}\Qp$ is induced from a character of an index 2 open subgroup of $G$;
\item $\rho$ is the twist with a character of a representation $\Pi\to\GL_2(\Qp)$ of finite image;
\item $\rho\otimes_{\Q_p}\Qp$ is reducible.
\end{enumerate}
\end{cor}


\begin{proof}
By Lemma \ref{lemma:lie}, the image $G=\rho(\Pi)$ is a $p$-adic Lie subgroup of $\GL_2(\Q_p)$. We apply Proposition \ref{prop:liealg} to it:
\begin{itemize}
\item if $\rho$ is strongly absolutely irreducible, we are in case ($i$) of Proposition \ref{prop:liealg}, hence ($i$) of Corollary \ref{cor:bigimage};
\item if $G$ is contained in the normalizer of a maximal torus $T(\Qp)$ in $\GL_2(\Qp)$, then the pre-image of $T(\Qp)$ under $\rho$ is an abelian subgroup $\Pi_0\subset\Pi$ of index 2, and by Frobenius reciprocity $\rho$ is induced from a character of $\Pi_0$;
\item if the projective image of $\rho$ is finite, then $G\cap\SL_2(\Q_p)$ is finite, and since $G\cap\SL_2(\Q_p)$ is always of finite index in the image of the constant-determinant twist $\rho_0$ of $\rho$, $\rho_0$ is of finite image;
\item if $G$ is contained in a Borel subgroup of $\GL_2(\Qp)$, then $\rho$ is not absolutely irreducible.
\end{itemize}
\end{proof}

\begin{rem}
Corollary \ref{cor:bigimage} implies the well-known fact that, if $\rho$ is absolutely irreducible and $\rho(\Pi)$ admits an abelian subgroup of finite index, but not a scalar subgroup of finite index, then $\rho$ is induced \cite[Theorem 2.3]{Ribet1975}.
\end{rem}

We also record the following lemma, that allows one to rule out some cases of Corollary \ref{cor:bigimage} in the applications. 

\begin{lemma}\label{lemma:noind}
Assume that $\rho$ is irreducible and $\ovl\rho$ is reducible. Then:
\begin{enumerate}[label=(\roman*)]
\item $\rho$ is not the twist with a character of a representation of finite image;
\item if $\rho$ is induced, then the projective image of $\ovl\rho^\sms(\Pi)$ is of order exactly 2;
\item if $\rho$ is absolutely irreducible and the projective image of $\ovl\rho^\sms$ is of order $>2$, then the image of $\rho$ contains an open subgroup of $\SL_2(\Z_p)$.
\end{enumerate}
\end{lemma}

\begin{proof}\mbox{ }
\begin{enumerate}[label=\textit{(\roman*)}]
\item Assume that $\rho$ is the twist with a character of a representation $\rho_0$ of finite image. Since $\ovl\rho$ is reducible, so is $\ovl\rho_0$. The kernel of $\rho_0(\Pi)\to\ovl\rho_0(\Pi)$ is a finite subgroup of $\wtl\Gamma_2(p)$, hence trivial. In particular, $\rho_0$ is itself reducible, and so its twist $\rho$ cannot be irreducible, contradicting our hypothesis.
\item Assume that $\rho$ is induced by a character of an index 2 subgroup $\Pi^\prime$ of $\Pi$, and let $c\in\Pi$ be any lift of a generator of $\Pi/\Pi^\prime$. Then, after suitable extension of the coefficients and conjugation, $\rho$ maps any element of $\Pi^\prime$ to a diagonal matrix and $c$ to an antidiagonal matrix. If there exists $d\in\Pi^\prime$ such that $\ovl\rho(d)$ is a non-scalar (diagonal) matrix, then $\ovl\rho(d)$ and $\ovl\rho(c)$ can never be upper triangular in the same basis, so that $\ovl\rho$ cannot be reducible. If instead $\ovl\rho(\Pi^\prime)$ consists of scalar matrices (which is equivalent to the projective image of $\ovl\rho$ being of order 2, since $\ovl\rho(c)$ is not a scalar), then any $\F_p$-basis consisting of $\ovl\rho(c)$-eigenvalues diagonalizes $\ovl\rho$.
\item This is an immediate consequence of \textit{(i)} and \textit{(ii)} together with Corollary \ref{cor:bigimage}.
\end{enumerate}
\end{proof}

\subsection{The case of arbitrary dimension}

Now let $d\ge 1$ be arbitrary. In this generality, we do not have as simple a classification as in Proposition \ref{prop:liealg}. We simply record here some facts to be used in the main text.
\begin{lemma}\label{lemma:SLGL} Let $G$ be a closed subgroup of $\GL_d(\ZZ_p)$. Assume that:
\begin{itemize}
\item[a)] $G\cap\SL_d(\ZZ_p)$ is an open subgroup of $\SL_d(\ZZ_p)$;
\item[b)] $\det G$ is an open subgroup of $\ZZ_p^\times$. 
\end{itemize}
Then  $G$ is open in $\GL_d(\ZZ_p)$. \\
Conversely, if $G$ is open in $\GL_d(\ZZ_p)$, then $a)$ and $b)$ hold.
\end{lemma}

\begin{proof}
By $a$), there exists a positive integer $M$ such that $\alpha^M\in G$, for every $\alpha\in \SL_d(\ZZ_p)$. Let $\mathcal{S}$ (resp. $\mathcal{S}(p^n)$)  be the set of scalar matrices in $\GL_d(\ZZ_p)$  (resp. in $\wtl\Gamma_d(p^n)$). We first show  that $G\cap \mathcal{S}$ is open in $\calS$. By hypothesis $b$), there exists an integer $n$ such that $1+p^n\ZZ_p\subseteq \det G$. Let $\alpha\in\calS(p^n)$; since $\det\alpha\in 1+p^n\ZZ_p$, there exists $\beta\in G$ such that $\alpha\beta^{-1}\in\SL_d(\ZZ_p)$. Then $(\alpha\beta^{-1})^M=\alpha^M\beta^{-M}\in G $ and this implies that $\alpha^M\in G$. Therefore $G\cap\calS$ contains $\calS(p^n)^M$, which is open in $\calS$. Let $m\geq n$ be such that $\calS(p^m)\subseteq G$. Let $\gamma\in \wtl\Gamma_d(p^m)$; since $\det \calS(p^m)=(1+p^m\ZZ_p)^2=1+p^m\ZZ_p $, there exists $\delta\in \calS(p^m)$ such that $\gamma\delta^{-1}\in \SL_d(\ZZ_p)$, so that again we argue that $\gamma^M\delta^{-M}\in G$ and $\gamma^M\in G$. Then $G$ contains $\wtl\Gamma_d(p^m)^M$, which is open in $\GL_d(\ZZ_p)$.\\
Conversely, assume that $G$ is open in $\GL_2(\ZZ_p)$; then $\wtl\Gamma_d(p^n)\subseteq G$ for some $n$, so that  $a)$ and $b)$ are obvious.
\end{proof}

\begin{remark}\label{rem:opendet}
The subgroup $\det G\subset\Z_p^\times$ is open if and only if $\det G$ is infinite: indeed, if $\Z_p^\times$ is equipped with the $\Z_p$-module structure given by exponentiation, any $\Z_p$-submodule of $\Z_p^\times$ of rank 1 is of finite index as a subgroup of $\Z_p^\times$, hence open.
 \end{remark}
\bigskip
\begin{prop}\label{prop:riehm}\mbox{ }
\begin{enumerate}[label=(\roman*)]
\item Let $G$ be an open subgroup of $\SL_d(\Z_p)$, and let $H$ be a normal subgroup of $G$. Then, either $H$ consists of scalar matrices, or $H$ is open in $G$.
\item Let $G$ be an open subgroup of $\GL_d(\Z_p)$, and let $H$ be a normal subgroup of $G$ such that $\det H$ is infinite. Then either $H$ admits an open subgroup consisting of scalar matrices, or $H$ is open in $\GL_d(\Z_p)$.
\end{enumerate}
\end{prop}

\begin{proof}
Since $\SL_d$ is a semisimple algebraic group (over $\Q_p$) with a simple Lie algebra, statement ($i$) follows from the Theorem in \cite[Appendix]{RiehmNorm1}.

For ($ii$), consider the intersections $G_1=G\cap \wtl\Gamma_d(p^{v_p(d)})$ and $H_1=H\cap \wtl\Gamma_d(p^{v_p(d)})$, which are open in $G$ and $H$, respectively. Clearly $H_1$ is normal in $G_1$. As in Section \ref{sec:constdet}, we define the constant-determinant twists $G_0$ and $H_0$ of $G_1$ and $H_1$, respectively. Since $G_1\subset\wtl\Gamma_d(p^{v_p(d)})$, we have $G_0\subset\SL_d(\Z_p)$. Proposition \ref{prop:twistfull} implies that $G_0$ is open in $\SL_d(\Z_p)$, and an easy calculation shows that $H_0$ is a normal subgroup of $G_0$. Then part (i) implies that $H_0$ is either open in $G_0$, hence in $\SL_d(\Z_p)$, or it consists of scalar matrices. In the first case, $H_1$ contains an open subgroup of $\SL_d(\Z_p)$ by Proposition \ref{prop:twistfull}, hence of $\GL_d(\Z_p)$ by Lemma \ref{lemma:SLGL}, since $\det(H_1)$ is infinite. In the second case, $H_1$ is an open subgroup of $H$ consisting of scalar matrices: indeed, every element of $H_0$ is the product of an element of $H_1$ with a scalar matrix, so that $H_0$ consists of scalar matrices if and only if $H_1$ does. 
%
\end{proof}

 Now let $\KK$ be a number field, $v$ the $p$-adic place of $\KK$ determined by the fixed embedding $\QQ\into\Qp$, and $G_v\subset G_{\KK}$ the corresponding decomposition group at $v$. As usual, $I_v$ denotes the inertia subgroup of $G_v$ and $I_v^w$ the wild inertia subgroup of $I_v$. We apply the above considerations to the case where $G=\rho(G_v)$ for a representation $\rho:G_v\to\GL_d(\ZZ_p)$.
The next fact is a direct consequence of Remark \ref{rem:opendet}:

\begin{lemma}\label{lem:determinant}  The following conditions  are equivalent:
\begin{itemize}
\item $\det\rho(I_v)$ is infinite;
    \item $\det\rho(I_v^w)\not=\{1\}$;
    \item $\det\rho(I_v)$ is open in $\ZZ_p^\times$.
    \end{itemize}
\end{lemma}

\begin{proposition}\label{prop:inaspettata} Assume that $\rho(G_v)$ contains an open subgroup of $\SL_d(\ZZ_p)$. Then:
\begin{enumerate}[label=(\roman*)] 
\item $\rho(I_v)$ contains an open subgroup of $\SL_d(\ZZ_p)$;
\item if $\det(\rho(I_v))$ is infinite, then $\rho(I_v)$ is open in $\GL_d(\ZZ_p)$.
\end{enumerate}
\end{proposition}

\begin{proof}\mbox{ }
\begin{enumerate}[label=(\textit{\roman*})]  \item  By Proposition \ref{prop:riehm} \textit{(i)} applied to $G=\rho(G_v)\cap\SL_d(\ZZ_p)$, $H=\rho(I_v)\cap \SL_d(\ZZ_p)$, we deduce that either $H$ consists of scalar matrices, or $H$ is open in $G$. In the second case we are done. In the first case, the elements of $H$ are scalar matrices $\alpha\cdot\Id$ of determinant $\alpha^d=1$, which forces $\alpha^{p-1}=1$ since $\alpha\in\Z_p^\times$. We have an exact sequence
\[1\rightarrow H\rightarrow \rho(G_v)\buildrel\theta\over\rightarrow \rho(G_v)/\rho(I_v)\times \ZZ_p^\times, \]
where $\theta(\sigma)=(\overline\sigma,\det\sigma)$, $\overline\sigma$ being the class of $\sigma$ in $\rho(G_v)/\rho(I_v)$. By hypothesis, there exists $n$ such that $\Gamma_d(p^n)\subseteq \rho(G_v)$; since $\Gamma_d(p^n)$ is a pro-$p$-group, $\Gamma_d(p^n)\cap H=\{1\}$. Then $\theta$ is injective on $\Gamma_d(p^n)$. But this is a contradiction because the image of $\theta$ is an abelian group.
\item If $\det(\rho(I_v))$ is infinite, then by Lemma \ref{lem:determinant}, $\det(\rho(I_v^w))$ is open in $\ZZ_p^\times$. It follows from Lemma \ref{lemma:SLGL} that $\rho(I_v)$ is open in $\GL_d(\ZZ_p)$.\end{enumerate}
\end{proof}

\subsection*{Acknowledgments} A.C. is supported by the DFG grant CRC 326 GAUS. He would like to thank the University of Luxembourg for the possibility of inviting L.T. for a visit that marked the start of the project. \\
L.T. is a member of the INDAM group GNSAGA.\\
The authors thank Francesco  Amoroso, Sara Checcoli, Davide Lombardo, Christian Maire, Marzio Mula and Riccardo Pengo for fruitful discussions and suggestions.
They are grateful to  Jean-Pierre Serre for pointing out the correct reference for Lemma \ref{lem:berger}.

\bigskip

\bibliographystyle{abbrv}
\bibliography{BOG}

\bigskip

\begin{small}
\noindent\textsc{Andrea Conti} -- IWR, Heidelberg University, Im Neuenheimer Feld 205, 69120 Heidelberg, Germany, \url{andrea.conti@iwr.uni-heidelberg.de}

\smallskip

\noindent\textsc{Lea Terracini} -- Dipartimento di Informatica, Università di Torino, Corso Svizzera 185, 10149 Torino, Italy, \url{lea.terracini@unito.it}
\end{small}

\end{document}